\documentclass[11pt]{amsart}
\usepackage{amsfonts, amstext, amsmath, amsthm, amscd, amssymb, upgreek}
\usepackage{graphicx, color,  subfigure, wrapfig, overpic, url}
\usepackage[dvipsnames]{xcolor}
\usepackage{tikz, url}
\AtBeginDocument{
   \def\MR#1{}
}

\newcommand{\C}{\mathbb{C}}
\newcommand{\R}{\mathbb{R}}
\newcommand{\Q}{\mathbb{Q}}
\newcommand{\Z}{\mathbb{Z}}

\renewcommand{\P}{\mathcal{P}}

\newcommand{\W}{\mathcal{W}}
\newcommand{\vol}{{\rm vol}}

\newcommand{\voct}{{v_{\rm oct}}}
\newcommand{\vtet}{{v_{\rm tet}}}
\newcommand{\LL}{\mathcal{L}}
\newcommand{\CC}{\mathcal{C}}
\newcommand{\RR}{\mathcal{R}}
\newcommand{\KK}{\mathcal{K}}
\newcommand{\K}{\upkappa}
\newcommand{\toF}{\stackrel{\rm{F}}{\to}}

\newcommand{\volbp}{{\rm vol}^{\lozenge}}
\newcommand{\G}{\tilde G}

 \newcommand{\TT}{\mathbb{T}}
\newcommand{\m}{\mathrm{m}}

\newcommand{\LLi}{\mathrm{Li}} 
\newcommand{\re}{\mathop{\mathrm{Re}}} 
\newcommand{\im}{\mathop{\mathrm{Im}}} 

\newtheorem{theorem}{Theorem}
\newtheorem{corollary}[theorem]{Corollary}
\newtheorem{lemma}[theorem]{Lemma}

\newtheorem{conjecture}[theorem]{Conjecture}

\newtheorem*{namedtheorem}{\theoremname}
\newcommand{\theoremname}{testing}

\theoremstyle{definition}
\newtheorem{definition}[theorem]{Definition}

\newtheorem{remark}[theorem]{Remark}

\title[Mahler Measure and the Vol-Det Conjecture]
{Mahler Measure and the Vol-Det Conjecture}
\author[A.\ Champanerkar]{Abhijit Champanerkar}
\address{Department of Mathematics, College of Staten Island \& The Graduate Center, City University of New York, New York, NY}
\email{abhijit@math.csi.cuny.edu}
\author[I. \ Kofman]{Ilya Kofman}
\address{Department of Mathematics, College of Staten Island \& The Graduate Center, City University of New York, New York, NY}
\email{ikofman@math.csi.cuny.edu}
\author[M. \ Lal\'in]{Matilde Lal\'in}
\address{Universit\'e de Montr\'eal, Pavillon Andr\'e-Aisenstadt, D\'ept. de math\'ematiques et de statistique, CP 6128, succ. Centre-ville Montr\'eal, Qu\'ebec, H3C 3J7, Canada}
\email{mlalin@dms.umontreal.ca}

\begin{document}

\maketitle

\begin{abstract}
The Vol-Det Conjecture relates the volume and the determinant of a
hyperbolic alternating link in $S^3$.  We use exact computations of
Mahler measures of two-variable polynomials to prove the Vol-Det
Conjecture for many infinite families of alternating links.

We conjecture a new lower bound for the Mahler measure of certain
two-variable polynomials in terms of volumes of hyperbolic regular
ideal bipyramids.  Associating each polynomial to a toroidal link
using the toroidal dimer model, we show that every polynomial
which satisfies this conjecture with a strict inequality gives rise to
many infinite families of alternating links satisfying the Vol-Det
Conjecture.  We prove this new conjecture for six toroidal links by
rigorously computing the Mahler measures of their two-variable polynomials.
\end{abstract}

\section{Introduction}
The deep connections between the Mahler measure of two-variable
polynomials and hyperbolic volume have been investigated by several
authors (see, e.g., \cite{BoydRV2002,brv03,boyd02,kenyonlap,Lalin,GM}).
The following examples illustrate some of the remarkable
relationships that have been discovered: Let $K$ be the figure-eight
knot, with $A$-polynomial $A(L,M)$ \cite{CCGLS}, and let $p(z,w)$ be
the characteristic polynomial of the toroidal dimer model on the
hexagonal lattice \cite{kenyon-survey-1}. Let $\m(P)$ denote the
logarithmic Mahler measure of a two-variable polynomial $P$,
and let $\vol(K)$ denote the hyperbolic volume of $S^3-K$.  Then
\begin{align}
& \vol(K) = 2\pi\, \m(1+x+y) = \frac{3\sqrt{3}}{2}L(\chi_{-3},2), \label{eq:3mm1} \\
& \vol(K) = \pi\, \m(A(L,M)) = \pi\, \m(M^4 + L(1-M^2-2M^4-M^6+M^8)-L^2M^4),  \label{eq:3mm2} \\
& \vol(K) = \frac{2\pi}{5}\, \m(p(z,w)) = \frac{2\pi}{5}\, \m\left(6-w-\frac{1}{w}-z-\frac{1}{z}-\frac{z}{w}-\frac{w}{z}\right).  \label{eq:3mm3} 
\end{align}

Equation~(\ref{eq:3mm1}), a famous result of Smyth \cite{Smyth}, was the
first instance where Mahler measure, hyperbolic volume and
special values of $L$-functions were related.  Equation~(\ref{eq:3mm2}), discovered by
Boyd \cite{boyd02} and later generalized by Boyd and Rodriguez-Villegas
\cite{BoydRV2002,brv03}, is an example of how Mahler measures of $A$-polynomials, which are
invariants of cusped hyperbolic 3-manifolds, are related to sums of
hyperbolic volumes of 3-manifolds using regulators on algebraic
curves.  Equation~(\ref{eq:3mm3}), discovered by Kenyon, arose
from his study of the entropy of toroidal dimer models \cite{kenyon-survey-1}.

The Vol-Det Conjecture relates the volume and determinant of a
hyperbolic alternating link in $S^3$.  In this paper, we use exact
computations of Mahler measures of two-variable polynomials to prove
the Vol-Det Conjecture for many infinite families of alternating
links.  Specifically, we formulate a conjectured inequality for
toroidal links (Conjecture~\ref{conj:volbipy} below) that relates
hyperbolic geometry, Mahler measure and toroidal dimer models.  We
then prove that every toroidal link which satisfies
Conjecture~\ref{conj:volbipy} with a strict inequality gives rise to
many infinite families of alternating links satisfying the Vol-Det
Conjecture.  We prove Conjecture~\ref{conj:volbipy} for six toroidal
links by explicitly computing the Mahler measures of two variable
polynomials using a technique developed by Boyd and
Rodriguez-Villegas. In particular, we give the complete proof of equation~\eqref{eq:3mm3}
above.
The motivation for Conjecture~\ref{conj:volbipy} came from studying the hyperbolic
geometry of biperiodic alternating links in \cite{ckp:bal_vol}.

\subsection{Main Conjecture}

Let $I=(-1,1)$. Let $L$ be a link in the thickened torus $T^2\times I$
with an alternating diagram on $T^2\times\{0\}$, projected onto the
$4$--valent graph $G(L)$.  The diagram is \emph{cellular} if the
complementary regions are disks, which are called the \emph{faces} of
$L$ or of $G(L)$.  
When lifted to the universal cover of $T^2\times I$, the link $L$ 
becomes a \emph{biperiodic alternating link} $\LL$ in
$\R^2\times I$,
such that $L=\LL/\Lambda$ for a two-dimensional lattice
$\Lambda$ acting by translations of $\R^2$.
We will refer to $\LL$ as a link, even though it has infinitely many components homeomorphic to $\R$ or $S^1$.  
The faces of $\LL$ are the
complementary regions of its diagram in $\R^2$, which are the regions
$\R^2-G(\LL)$.  The diagram of $L$ on $T^2\times\{0\}$ is
\emph{reduced} if four distinct faces meet at every crossing of
$G(\LL)$ in $\R^2$.  Let $c(L)$ denote the crossing number of the
reduced alternating projection of $L$ on $T^2\times \{0\}$, which is
minimal by \cite{Adams:crossings}.  Throughout the paper, link
diagrams on $T^2\times \{0\}$ will be alternating, reduced and
cellular.

Let $B_n$ denote the hyperbolic regular ideal bipyramid whose link
polygons at the two coning vertices are regular $n$--gons.
The hyperbolic volume of $B_n$ is given by
$$ \vol(B_n) = n \left(\int_0^{2\pi/n} -\log|2\sin(\theta)|d\theta + \int_0^{\pi(n-2)/2n} -2\log|2\sin(\theta)|d\theta \right). $$
See~\cite{Adams:bipyramids} for more details and a table of values of $\vol(B_n)$.  
If we let $n=2$, note that $\vol(B_2)=0$.

For a face $f$ of a planar or toroidal graph, let $|f|$ denote the
degree of the face; i.e., the number of its edges.
Let $L$ be an alternating link diagram on the torus as above.
Define the \emph{bipyramid volume} of $L$ as follows:
$$ \volbp(L) = \sum_{f \in \{\text{faces of} \ L\}}\vol(B_{|f|}). $$

For a biperiodic alternating link $\LL$ in $\R^2\times I$, the
projection graph $G(\LL)$ in $\R^2$ is biperiodic and can be
checkerboard colored.
The {\em Tait graph} $G_{\LL}$ is the planar checkerboard graph for
which a vertex is assigned to every shaded region and an edge to every
crossing of $\LL$.
Using the other checkerboard coloring yields the dual graph $G_{\LL}^*$. 
We form the bipartite {\em overlaid  graph} $G_{\LL}^b = G_{\LL}\cup G_{\LL}^*$ determined by the link diagram of $\LL$ in $\R^2$
as follows: The black vertices of $G_{\LL}^b$ are the vertices of
$G_{\LL}$ and of $G_{\LL}^*$; the white vertices of $G_{\LL}^b$ are the crossings of $\LL$.
The edges of $G_{\LL}^b$ join a black vertex for each face
of $\LL$ to every white vertex incident to the face.
The overlaid graph $G_{\LL}^b$ is a biperiodic {\em balanced bipartite} graph; i.e., the number of black
vertices equals the number of white vertices in a fundamental domain.
The $\Lambda$--quotient of $G_{\LL}^b$ is the toroidal graph $G^b_L$,
which is also a balanced bipartite graph. See Figures \ref{fig-square}
and \ref{fig:triaxial}.

This makes it possible to define the toroidal dimer model on
$G^b_{L}$.  A dimer covering of a graph is a subset of edges that
covers all the vertices exactly once, so each vertex is the endpoint
of a unique edge.  The toroidal dimer model on $G^b_{L}$ is a statistical
mechanics model of the set of dimer coverings of $G^b_{L}$.  The {\em
  characteristic polynomial} of the dimer model is defined as $p(z,w)
= \det \K(z,w),$ where $\K(z,w)$ is the weighted, signed adjacency matrix with
rows indexed by black vertices and columns by white vertices, and
matrix entries determined by a certain choice of signs on edges, and a
choice of homology basis for the $\Lambda$--action.
See Section~\ref{sec:background} and
\cite{cimasoni-survey,kenyon-survey-1,ck:det_mp} for details and
examples.

Let $G^b_n$ be the finite balanced bipartite toroidal graph $G^b_{\LL}/(n\Lambda)$.
Let $Z(G^b_n)$ be the number of dimer coverings of $G^b_n$.
Kenyon, Okounkov and Sheffield \cite{KOS} gave an explicit expression
for the asymptotic growth rate of the toroidal dimer model on $\{G^b_n\}$:
$$ \log Z(G^b_{\LL}) := \lim_{n\to\infty}\frac{1}{n^2}\log Z(G^b_n) = \m(p(z,w)).$$
The number $Z(G^b_{\LL})$ is called the {\em partition function}, and
the limit is the {\em entropy} of the toroidal dimer model.
It is proved in \cite{KOS} that the Mahler measure of the
characteristic polynomial is independent of the choices made to obtain
$\K(z,w)$, so the entropy is determined by $G^b_{\LL}$.

\begin{conjecture}[Main Conjecture]\label{conj:volbipy}
Let $\LL$ be any biperiodic alternating link, with toroidally alternating $\Lambda$--quotient link $L$. 
Let $p(z,w)$ be the characteristic polynomial of the toroidal dimer model on $G_{\LL}^b$.
Then
$$ \volbp(L) \leq 2\pi\, \m(p(z,w)). $$
\end{conjecture}

The link $L$ is often hyperbolic in $T^2\times I$; i.e., $(T^2\times I)-L$ is a complete finite-volume
hyperbolic $3$--manifold \cite{adams:hyperbolicity,ckp:bal_vol,HowiePurcell}.
In~\cite{ckp:bal_vol}, it was proved that
\begin{equation}\label{eq:volbp}
  \vol((T^2\times I)-L)\leq \volbp(L),
\end{equation}  
with equality for semi-regular links.  Thus,
Conjecture~\ref{conj:volbipy} would imply that
\begin{equation}\label{eq:volbp-3}
\vol((T^2\times I)-L)\leq \volbp(L) \leq 2\pi\, \m(p(z,w)).
\end{equation}  

In this paper, we prove Conjecture~\ref{conj:volbipy} for six biperiodic alternating links using 
rigorous computations for the Mahler measures of the corresponding $p(z,w)$.
Our examples include cases for which the expression~\eqref{eq:volbp-3} is sharp, with both equalities, and cases for which both are strict inequalities.
We now explain several results at the intersection of geometry,
topology and number theory implied by Conjecture~\ref{conj:volbipy},
which therefore hold in these special cases.

\subsection{Volume and determinant.}
The determinant of a knot is one of the oldest knot invariants that
can be directly computed from a knot diagram.  For any knot or link $K$,
$$ \det(K) = |\det(M + M^T)|  = |\Delta_K(-1)|= |V_K(-1)|, $$ 
where $M$ is any Seifert matrix of $K$, $\Delta_K(t)$ is the Alexander
polynomial and $V_K(t)$ is the Jones polynomial of $K$ (see, e.g.,
\cite{lickorish}).

Experimental evidence has long suggested a close relationship between
the volume and determinant of alternating knots
\cite{Dunfield_website,stoimenow}.  The following inequality was
conjectured in \cite{ckp:gmax}, and verified for all alternating knots
up to 16 crossings, weaving knots \cite{ckp:weaving} with hundreds
of crossings, all 2--bridge links and alternating closed 3--braids \cite{Burton}.

\begin{conjecture}[Vol-Det Conjecture \cite{ckp:gmax}]\label{conj:voldet}
For any alternating hyperbolic link $K$, 
\[ \vol(K) < 2\pi\log\det(K). \]
\end{conjecture}
\noindent
It was shown in \cite{ckp:gmax} that the constant $2\pi$ is sharp;
i.e., for any $\alpha < 2 \pi$, there exist alternating links for which $\vol(K) > \alpha\log\det(K)$.

In \cite{ck:det_mp,ckp:density,ckp:gmax}, biperiodic alternating
links were considered as limits of sequences of finite hyperbolic
links.  In Section~\ref{sec:background}, we define a natural notion
of convergence for a sequence of alternating links to a biperiodic
alternating link $\LL$, called {\em F{\o}lner convergence almost
  everywhere}, denoted by $K_n\toF\LL$.
It was proved in \cite{ck:det_mp} that for any sequence of alternating links $K_n$
that converge to a biperiodic alternating link $\LL$ in this sense, the
determinant densities of $K_n$ converge to the density of the Mahler measure of 
the characteristic polynomial $p(z,w)$  of the associated toroidal dimer model:
\[ K_n\toF \LL \quad \Longrightarrow \quad {\lim_{n\to\infty}\frac{\log\det(K_n)}{c(K_n)} = \frac{\m(p(z,w))}{c(L)}}. \]

The following theorem implies that whenever 
Conjecture~\ref{conj:volbipy} holds with a strict
inequality, we obtain many infinite families of knots that satisfy the
Vol-Det Conjecture (Conjecture~\ref{conj:voldet}).

\begin{theorem}\label{Thm:voldet_ineq}
  Let $\LL$ be any biperiodic alternating link, with toroidally alternating quotient link $L$.
  Let $p(z,w)$ be the characteristic polynomial of the associated toroidal dimer model.
  Let $K_n$ be alternating hyperbolic links such that \mbox{$\displaystyle K_n\toF \LL$}.
  If $\volbp(L)< 2\pi\,\m(p(z,w))$, then $\vol(K_n) < 2\pi\log\det(K_n)$ for almost all $n$.
\end{theorem}

Note that for any $\LL$ as in Theorem~\ref{Thm:voldet_ineq}, the
infinite families of knots or links satisfying the Vol-Det Conjecture include
almost all $K_n$ for {\em every} sequence \mbox{$\displaystyle K_n\toF \LL$}.

\subsection{Lower bounds for Mahler measure.}
Finding lower bounds for Mahler measure has intrigued mathematicians for more than 80 years. 
Kronecker's lemma implies that polynomials in $\Z[z]$ with $\m(p)=0$ are exactly products of cyclotomic polynomials and monomials. 
Lehmer \cite{Lehmer} first asked in 1933 whether there exists $\varepsilon>0$ such that for every $p(z) \in \Z[z]$ with $\m(p)>0$, it follows that $\m(p)>\varepsilon$.  
Lehmer's question remains open to this day, although there are several results on specific families of polynomials \cite{Breusch,Smyth1,BDM} and general 
lower bounds that depend on the degree of $p(z)$ \cite{Dobrowolski}. 

For any multivariable polynomial, Boyd and Lawton \cite{Boyd:speculations,Lawton} showed that its Mahler measure is given by a limit of Mahler measures of single variable polynomials. Therefore, in terms of Lehmer's question, a lower bound for 
single variable polynomials would automatically imply a lower bound for multivariable polynomials. Nevertheless, finding multivariable polynomials with low Mahler measure has also attracted interest and speculation \cite{Boyd:speculations}. 
Smyth \cite{Smyth-Kronecker} characterized multivariable polynomials with $\m(p)=0$, generalizing Kronecker's lemma.

For a two-variable polynomial $p(z,w)$, Smyth's proof involves the {\em Newton polygon} $\Delta(p)$ in $\R^2$, which is  
the convex hull of $\{(m,n)\in \Z^2\, |\, \mbox{ the coefficient of }z^mw^n \mbox{ in } p \mbox{ is non-zero}\}$.
For each side $\Delta_\ell$ of $\Delta(p)$, one can associate a one-variable polynomial $p_\ell$ whose coefficients are those of $p$ corresponding to the points on $\Delta_\ell$.  Smyth proved that for all $\Delta_\ell$,
\begin{equation}\label{eq:smyth}
\m(p_\ell)\leq \m(p).
\end{equation}

It is interesting to compare the bound in
Conjecture \ref{conj:volbipy} with Smyth's bound \eqref{eq:smyth}.
For the polynomials we consider in this paper,
Conjecture \ref{conj:volbipy} yields a much better bound, and it is
actually sharp in two examples, which are discussed in Section~\ref{sec:background}.
Let $\vtet\approx 1.0149$ be the
volume of the regular ideal tetrahedron, $\voct\approx 3.6638$ be the
volume of the regular ideal octahedron, and $v_{16}\approx 7.8549$ be
the volume of the regular ideal bipyramid $B_8$.
We consider the following polynomials, for which the results are summarized in the table below.
\begin{eqnarray*}
\P_1&=& 4 +(w+\frac{1}{w} +z +\frac{1}{z})\\
\P_2&=& 6-(w+\frac{1}{w} +z +\frac{1}{z}+\frac{w}{z}+\frac{z}{w} )\\
\P_3&=&- z(w^2-4w+1)+w^2+4w+1 \\
\P_4&=& (1+w^2)(1-z)^2-w(6+20z+6z^2)\\
\P_5&=& -w^2z^2 + 6w^2z + 6wz^2 - w^2 + 28wz - z^2 + 6w + 6z - 1
\end{eqnarray*}
\renewcommand{\arraystretch}{1.5}
\begin{tabular}{|l|l|l|l|}
\hline 
$p$ & $\m(p)$ & $\volbp(L)/2\pi$  & maximal $\m(p_\ell)$ \\
\hline
$\P_1$ &$\frac{2\voct}{2\pi}\approx 1.16624361$&$\frac{2\voct}{2\pi}\approx1.16624361$&$\m(z+1)=0$\\
$\P_2$ &$\frac{10\vtet}{2\pi}\approx1.61532973$&$\frac{10\vtet}{2\pi}\approx1.61532973$&$\m(z+1)=0$\\
$\P_3$ &$1.65546767$ &$\frac{10\vtet}{2\pi}\approx 1.61532973$& $\m(z^2+4z+1)\approx 1.31695789$\\
$\P_4$ &$2.79856868$&$\frac{10\vtet+2\voct}{2\pi}\approx 2.78157335$&$\m(z^2-6z+1)\approx 1.76274717$\\
$\P_5$ &$3.14673710$&$\frac{8\vtet+\voct+v_{16}}{2\pi} \approx 3.12553175$&$\m(z^2-6z+1)\approx 1.76274717$\\
\hline 
\end{tabular}
\renewcommand{\arraystretch}{1}


\subsection{A typical example for Conjecture~\ref{conj:volbipy}.}

Our proven examples are rather special because the characteristic
polynomials that lend themselves to the methods which allow us to
compute $\m(p)$ exactly seem to be special.  We pause here to
present a more typical but only numerically verified example for Conjecture~\ref{conj:volbipy}.

\begin{center}
\begin{figure}[h]
\includegraphics[height=2in]{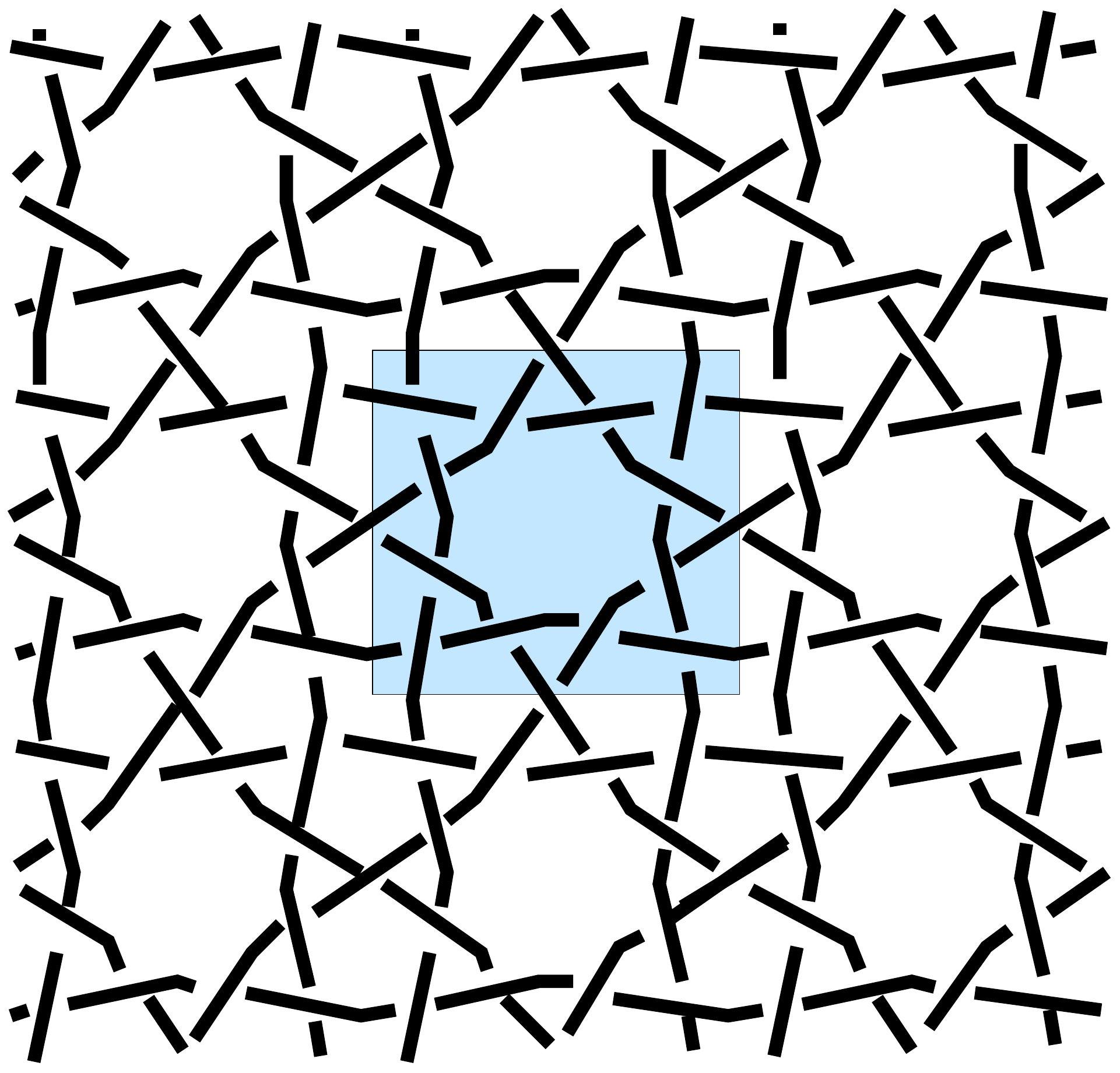}  
\caption{A typical biperiodic alternating link}
\label{fig:typical}
\end{figure}
\end{center}

Figure~\ref{fig:typical} shows the biperiodic alternating link $\LL$, and fundamental domain for its alternating quotient link $L$ in $T^2\times I$.
The fundamental domain for $L$ has one octagon, four pentagons, one square and eight triangles.  Thus, as $\vol(B_4)=\voct$ and $ \vol(B_3)=2\vtet$,
$$ \volbp(L) = \vol(B_8)+4\vol(B_5)+\voct+16\vtet \approx 47.704628. $$

Using SnapPy~\cite{snappy} inside Sage to verify the computation rigorously, we verified that
$$\vol((T^2\times I) - L) \approx 47.644829.$$ 

Using the method described in Section~\ref{sec:background}, we computed the following characteristic polynomial $p(z,w)$, which has genus $8$,
\begin{eqnarray*}
  p(z,w) &= & wz^2 + z^3 - 2wz + 104z^2 - 2z^3/w + w + 510z + 510z^2/w + z^3/w^2 - 2456z/w\\
  && + 104z^2/w^2 + 510/w + 1/z + 510z/w^2 + z^2/w^3 + 104/w^2 - 2/(wz) - 2z/w^3\\
  && + 1/w^3 + 1/(w^2z) + 104.\\
\end{eqnarray*}
Numerically, $2\pi\, \m(p)\approx 47.9214$, so $L$ satisfies Conjecture~\ref{conj:volbipy}, and inequality~\eqref{eq:volbp-3} within a range of $0.6\%$,
$$ \vol((T^2\times I) - L) < \volbp(L) < 2\pi\, \m(p). $$


\subsection{Organization}
In Section~\ref{sec:background}, we recall definitions, properties and
examples for the toroidal dimer model, F{\o}lner convergence of links,
Mahler measure and the Bloch-Wigner dilogarithm.  In
Section~\ref{sec:rth}, we prove Theorem~\ref{Thm:voldet_ineq}, as well
as its corollary, which gives a new bound on how much the volume of a
hyperbolic alternating link can change after drilling out an augmented
unknot.
In Section~\ref{sec:examples},
we prove six special cases of Conjecture~\ref{conj:volbipy}, and
provide numerical evidence to support it.

\subsection*{Acknowledgements}
We thank the organizers of the workshop {\em Low-dimensional topology
and number theory} at MFO (Oberwolfach Research Institute for Mathematics), where this work was started.
The first two authors acknowledge support by the Simons Foundation and
PSC-CUNY. The third author was partially supported by the  Natural Sciences and Engineering Research Council
of Canada [Discovery Grant 355412-2013]. We thank the anonymous referee for careful and thoughtful revisions.


\section{Background}
\label{sec:background}

\subsection{Toroidal dimer model}
\label{sec:tdimer}

The study of the dimer model is an active research area (see the
excellent introductory lecture notes \cite{cimasoni-survey,kenyon-survey-1}).
As mentioned in the Introduction, a {\em dimer covering} (or {\em perfect matching}) of a graph is a
pairing of adjacent vertices.  The dimer model on a graph $G$ is a
statistical mechanics model of the set of dimer coverings of $G$.

\subsubsection*{\bf Planar graphs}
Let $G$ be a finite balanced bipartite planar
graph, with edge weights $\mu_e$ for each edge $e$ in $G$.  The {\em
  Kasteleyn signs} are a choice of sign for each edge, such that
each face of $G$ with $0$ mod $4$ edges has an odd number of negative signs,
and each face with $2$ mod $4$ edges has an even number of negative signs.  A
{\em Kasteleyn matrix} $\K$ is a weighted, signed adjacency matrix of
$G$, such that rows are indexed by black vertices, and columns by
white vertices.  The matrix coefficients are $\pm\mu_e$, with the sign
given by the Kasteleyn sign on $e$.  Then, taking the sum over all
dimer coverings $M$ of $G$, the {\em partition function} $Z(G)$
satisfies (see \cite{cimasoni-survey,kenyon-survey-1}): 
$$ Z(G):= \sum_M \prod_{e\in M} \mu_e  =|\det \K|.$$  
With $\mu_e=1$ for every edge $e$, $Z(G)$ is the number of dimer coverings of $G$.  
Also see \cite{cdr-dimers} for relations between dimer coverings of planar graphs and knot theory.

\subsubsection*{\bf Toroidal graphs}
Now, let $G$ be a finite balanced bipartite toroidal graph.  As in the
planar case, we choose Kasteleyn signs on the edges of $G$.  We then
choose oriented simple closed curves $\gamma_z$ and $\gamma_w$ on
$T^2$, transverse to $G$, representing a basis of $H_1(T^2)$.  We
orient each edge $e$ of $G$ from its black vertex to its white vertex.
The weight on $e$ is
$$ \mu_e= z^{\gamma_z\cdot e}w^{\gamma_w\cdot e},$$ where $\cdot$
denotes the signed intersection number of $e$ with $\gamma_z$ or
$\gamma_w$.  For example, see Figure~\ref{square-fd1}.  The Kasteleyn
matrix $\K(z,w)$ is the weighted, signed adjacency matrix with rows indexed by
black vertices and columns by white vertices, and matrix entries
$\pm\mu_e$, with the sign given by the Kasteleyn sign on $e$.  The
{\em characteristic polynomial} is defined as
$$ p(z,w) = \det \K(z,w).$$
With $\mu_e$ as above, the number of dimer coverings of $G$ is given by (see \cite{cimasoni-survey,kenyon-survey-1}):
$$ Z(G) = \frac{1}{2} \ |-p(1,1) + p(-1,1) + p(1,-1) + p(-1,-1)|.$$

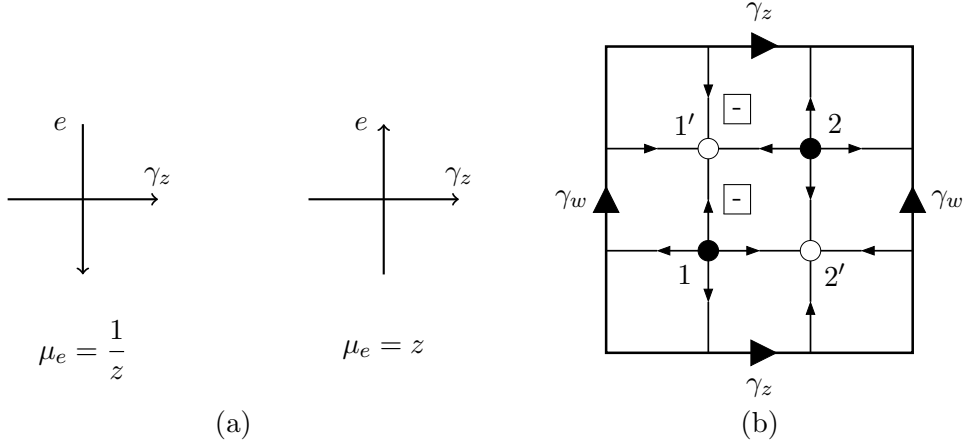
\begin{figure}

\begin{tikzpicture} 
\node at (0,0)
{ \includegraphics[height=1.75in]{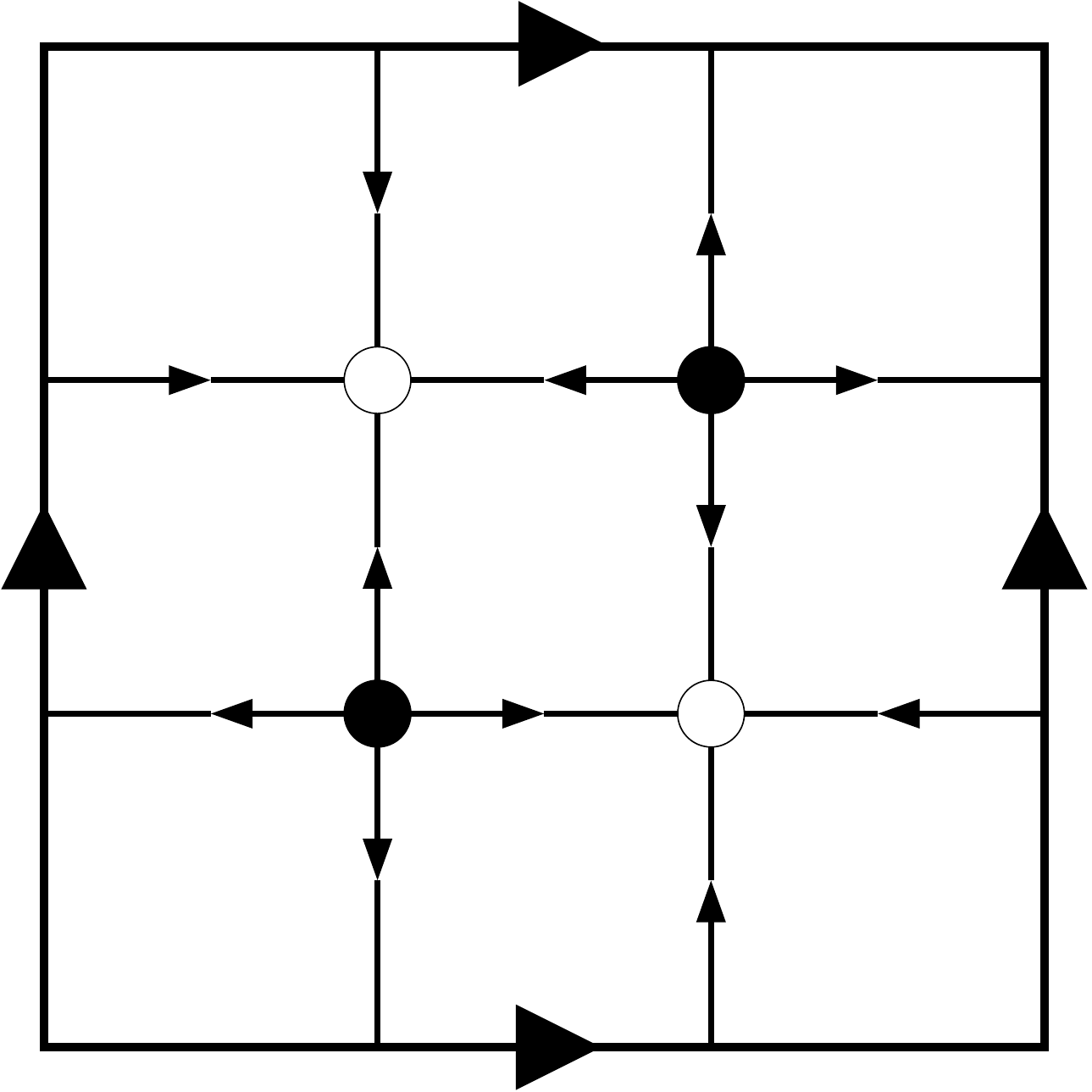}};
\node at (-2.5,0) {$\gamma_w$};
\node at (0,-2.5) {$\gamma_z$};
\node at (2.5,0) {$\gamma_w$};
\node at (0,2.5) {$\gamma_z$};
\node at (1,1) {2};
\node at (1,-1) {$2'$};
\node at (-1,1) {$1'$};
\node at (-1,-1) {1};
\node at (-0.3,0) {$\fbox{-}$};
\node at (-0.3,1.2) {$\fbox{-}$};

\draw[thick,->] (-6,0) -- (-4,0);
\draw[thick,->] (-5,-1) -- (-5,1);
\node at (-4,0.3) {$\gamma_z$};
\node at (-5.3,1) {$e$};
\node at (-5,-2) {$\mu_e=z$};

\draw[thick,->] (-10,0) -- (-8,0);
\draw[thick,<-] (-9,-1) -- (-9,1);
\node at (-8,0.3) {$\gamma_z$};
\node at (-9.3,1) {$e$};
\node at (-9,-2) {$\mu_e=\displaystyle{\frac{1}{z}}$};

\node at (-7,-3) {(a)};
\node at (0,-3) {(b)};

\end{tikzpicture}

\caption{(a) Edge weights $\mu_e= z^{\gamma_z\cdot e}$ to compute $\K(z,w)$. (b) Toroidal bipartite
  graph $G$ with a choice of Kasteleyn signs.}
\label{square-fd1}
\end{figure}

\subsubsection*{\bf Biperiodic graphs}
Let $G$ be a biperiodic bipartite planar graph, so that
translations by a two-dimensional lattice $\Lambda$ act by
isomorphisms of $G$.  Let $G_n$ be the finite balanced bipartite
toroidal graph given by the quotient $G/(n\Lambda)$.  Kenyon, Okounkov,
and Sheffield \cite{KOS} gave an explicit expression for the growth
rate of the toroidal dimer model on $\{G_n\}$:
\begin{theorem}\cite[Theorem~3.5]{KOS}\label{thm:KOS} Let $G$ be a biperiodic bipartite planar graph. Then
$$ \log Z(G) := \lim_{n\to\infty}\frac{1}{n^2}\log Z(G_n) = \m(p(z,w)).$$
\end{theorem}
Thus, Theorem~\ref{thm:KOS} says that, independent of any choice of
Kasteleyn signs and homology basis for the $\Lambda$--action, the
growth rate of any toroidal dimer model is given by the Mahler measure
of its characteristic polynomial.

In \cite{ck:det_mp}, the first two authors defined the following notion of convergence of
links in $S^3$ to a biperiodic alternating link.
\begin{definition}[\cite{ck:det_mp,ckp:gmax}]\label{def:folner_converge}
We will say that a sequence of alternating links $K_n$ {\em F{\o}lner converges almost everywhere} to the biperiodic alternating link $\LL$, denoted by $K_n\toF\LL$,
if the respective projection graphs $\{G(K_n)\}$ and $G(\LL)$ satisfy the following:  There are subgraphs $G_n\subset G(K_n)$ such that 
\begin{enumerate}
\setlength\itemsep{0.5em}
\item[($i$)] $ G_n\subset G_{n+1}$, and $\bigcup G_n=G(\LL)$, 
\item[($ii$)] $\lim\limits_{n\to\infty}|\partial G_n|/|G_n|=0$, where $|\cdot|$ denotes number of vertices, and $\partial G_n\subset G(\LL)$ consists of the vertices of $G_n$ that share an edge in $G(\LL)$ with a vertex not in $G_n$,
\item[($iii$)] $ G_n\subset G(\LL)\cap(n\Lambda)$, where $n\Lambda$ represents $n^2$ copies of the $\Lambda$--fundamental domain for the lattice $\Lambda$ such that $L=\LL/\Lambda$,
\item[($iv$)] $ \lim\limits_{n\to\infty} |G_n|/ c(K_n) = 1$, where $c(K_n)$ denotes the crossing number of $K_n$.
\end{enumerate}
\end{definition}

\begin{theorem}\cite{ck:det_mp}\label{Thm:det_mp}
Let $\LL$ be any biperiodic alternating link, with toroidally alternating quotient link $L$. 
Let $p(z,w)$ be the characteristic polynomial of the associated toroidal dimer model.
Let $K_n$ be alternating links such that \mbox{$\displaystyle K_n\toF \LL$}.
Then
\[ \lim_{n\to\infty}\frac{\log\det(K_n)}{c(K_n)} = \frac{\m(p(z,w))}{c(L)}. \]
\end{theorem}

Finally, all of our examples of biperiodic alternating links below satisfy the
hypotheses of \cite[Theorem~7.5]{ckp:bal_vol}, which implies that the
link diagram admits an embedding into $\R^2$ for which the faces are
cyclic polygons.  Such nice geometry allows us to draw the diagrams
for their overlaid graphs with vertices at the centers of the
corresponding circles.

\subsubsection*{\bf Example 1: Square weave}

\begin{figure}[h] 
\begin{tabular}{cc}
 \includegraphics[height=1.5in]{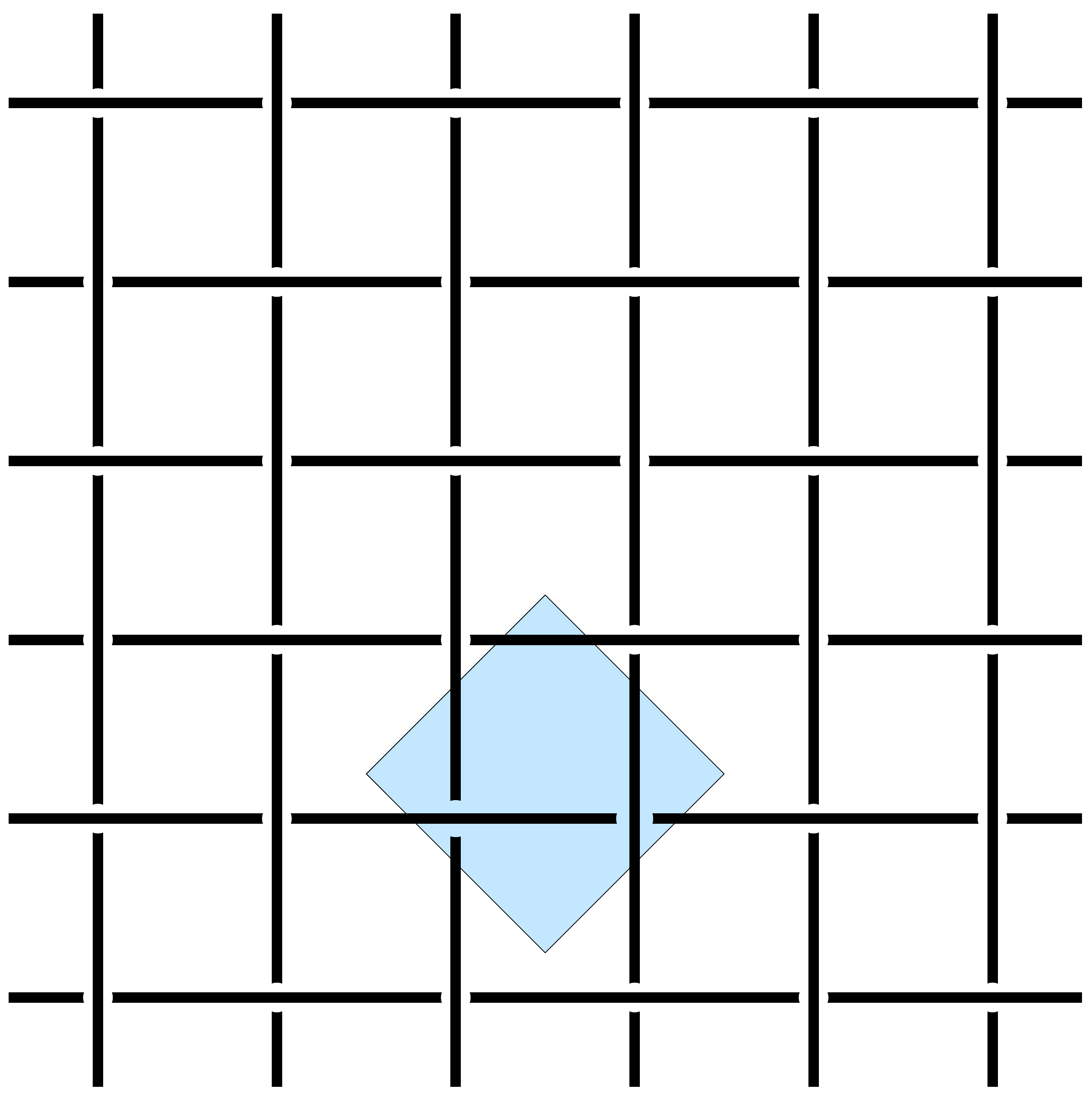} &  \quad \quad \quad \quad 
 \includegraphics[height=1.5in]{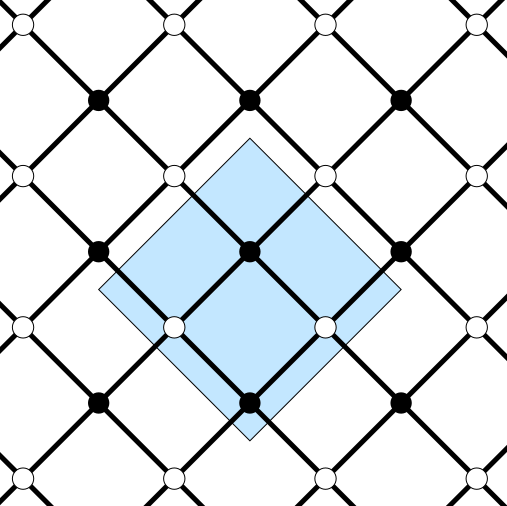}    \\
  (a) & \quad  \quad \quad \quad  (b)   \\
\end{tabular}
\caption{(a) Infinite square weave $\W$ and fundamental domain for $W$.
  (b) Overlaid graph $G_{\W}^b$ and fundamental domain for $G_W^b$.}
\label{fig-square}
\end{figure}

Figure~\ref{fig-square}(a) shows the infinite square weave $\W$, with a choice of fundamental domain, giving a toroidally alternating link $W$ with $c(W)=2$.
Both of the Tait graphs of $\W$ are the infinite square grid.
The overlaid graph $G_{\W}^b$ is shown in Figure~\ref{fig-square}(b), 
with the fundamental domain for $G_{W}^b$, which matches the toroidal graph shown in
Figure \ref{square-fd1}(b).

We can now compute $p(z,w)= \det\K(z,w)$ for $G=G_{W}^b$, as described above, and in more detail in \cite{cimasoni-survey,kenyon-survey-1}.
Using Figure~\ref{square-fd1}(b) with the ordering as shown, 
\begin{equation}
\label{eqn:charpol-weave} 
\K(z,w)=
\begin{bmatrix}
-1 - 1/z& 1 + w\cr
1+1/w& 1+z
\end{bmatrix},
\quad p(z,w) = -\left(4 + \frac{1}{w} + w + \frac{1}{z} +z\right).
\end{equation}
By Theorem~\ref{thm:mm-weave} below,
$\displaystyle 2\pi\, \m(p(z,w)) =  2\,\voct$.
By \cite{ck:det_mp,ckp:gmax}, it follows that for $K_n\toF\W$,
$\voct\approx 3.66386$ is the limit of both determinant densities and volume densities:
$$\lim_{n\to \infty} \frac{2 \pi \log \det(K_n)}{c(K_n)}
=\frac{2 \pi \m( p(z,w))}{c(W)}
=\voct 
=\lim_{n\to \infty} \frac{\vol(K_n)}{c(K_n)}.
$$

\subsubsection*{\bf Example 2: Triaxial link}

Figure~\ref{fig:triaxial}(a) shows part of the biperiodic alternating
diagram of the triaxial link $\LL$, and the fundamental domain for the
toroidally alternating link $L$ with $c(L)=3$.  Its projection graph
$G(\LL)$ is the trihexagonal tiling.  The Tait graphs of $\LL$ are the
regular hexagonal and triangular tilings, which form the biperiodic balanced bipartite overlaid
graph $G_{\LL}^b$, shown in Figure~\ref{fig:triaxial}(b).

We can now compute $p(z,w)= \det\K(z,w)$ for $G=G_{L}^b$, as above.
Using Figure~\ref{fig:triaxial}(c),
with the homology basis, ordered vertices and a choice of Kasteleyn signs on edges as shown,

\begin{equation}\label{eqn:charpol-triaxial}
\K(z,w)=
\begin{bmatrix}
1 & z & w \cr
1& 1& 1 \cr
1/z-1/w & 1/w-1 & 1-1/z 
\end{bmatrix}, 
\quad p(z,w)= 6-\left(\frac{1}{w}+w+\frac{1}{z}+z+\frac{w}{z}+\frac{z}{w}\right).
\end{equation}
By Theorem~\ref{thm:mm-triaxial} below,
$\displaystyle 2\pi\,\m(p(z,w)) = 10\vtet$, where
$\vtet\approx 1.01494$.
By \cite{ck:det_mp}, for $K_n\toF\LL$,
$$\lim_{n\to \infty} \frac{2 \pi \log \det(K_n)}{c(K_n)}
=\frac{2 \pi \m( p(z,w))}{c(L)}
=\frac{10\vtet}{3}. $$
Moreover, by \cite{ckp:bal_vol}, 
$$ \lim_{n\to \infty} \frac{\vol(K_n)}{c(K_n)} = \frac{\vol(T^2 \times I - L)}{c(L)} = \frac{10\vtet}{3}.$$

For the square weave and the triaxial link, the volume and
determinant densities both converge to the volume density of the
toroidal link, but we do not know of any other such examples.
The strict inequality satisfied by all the other examples in Section~\ref{sec:examples} seems to be more typical.
The first two authors and Purcell compute the exact hyperbolic volume of infinitely many other such biperiodic alternating links in \cite{ckp:bal_vol}.

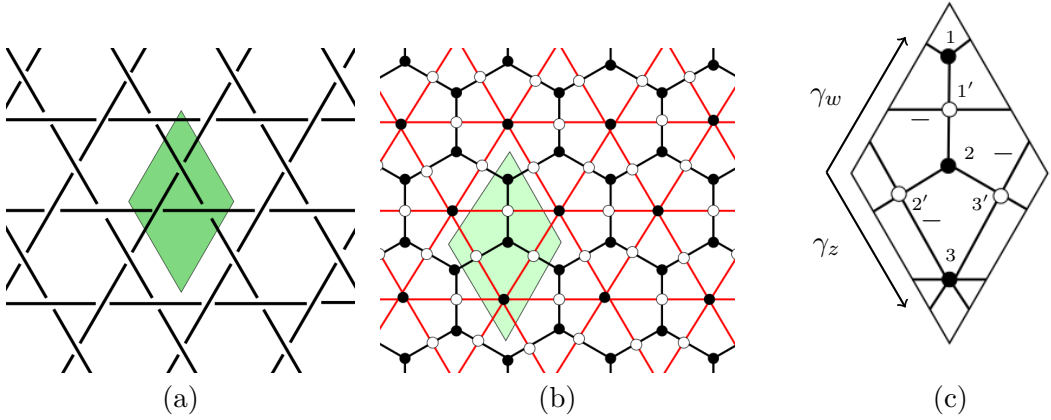
\begin{figure}
\begin{tabular}{ccc}
\includegraphics[height=1.7in]{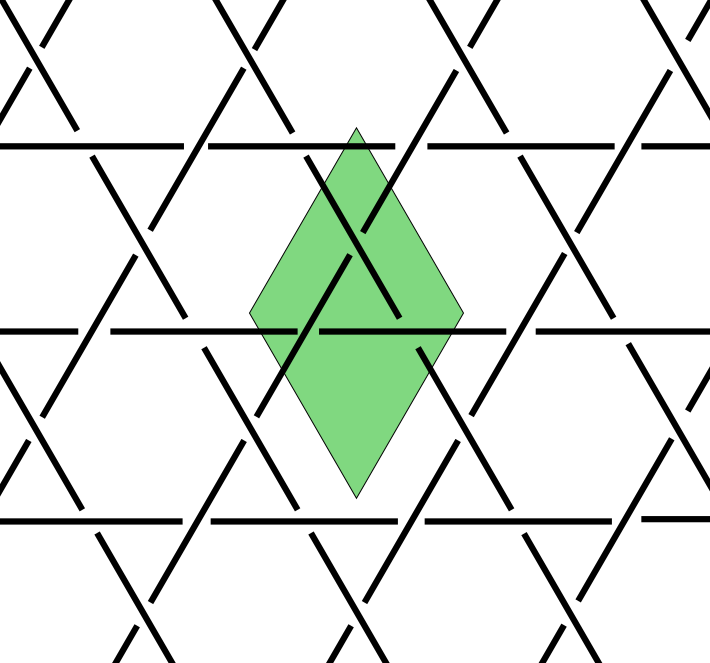} &
\includegraphics[height=1.7in]{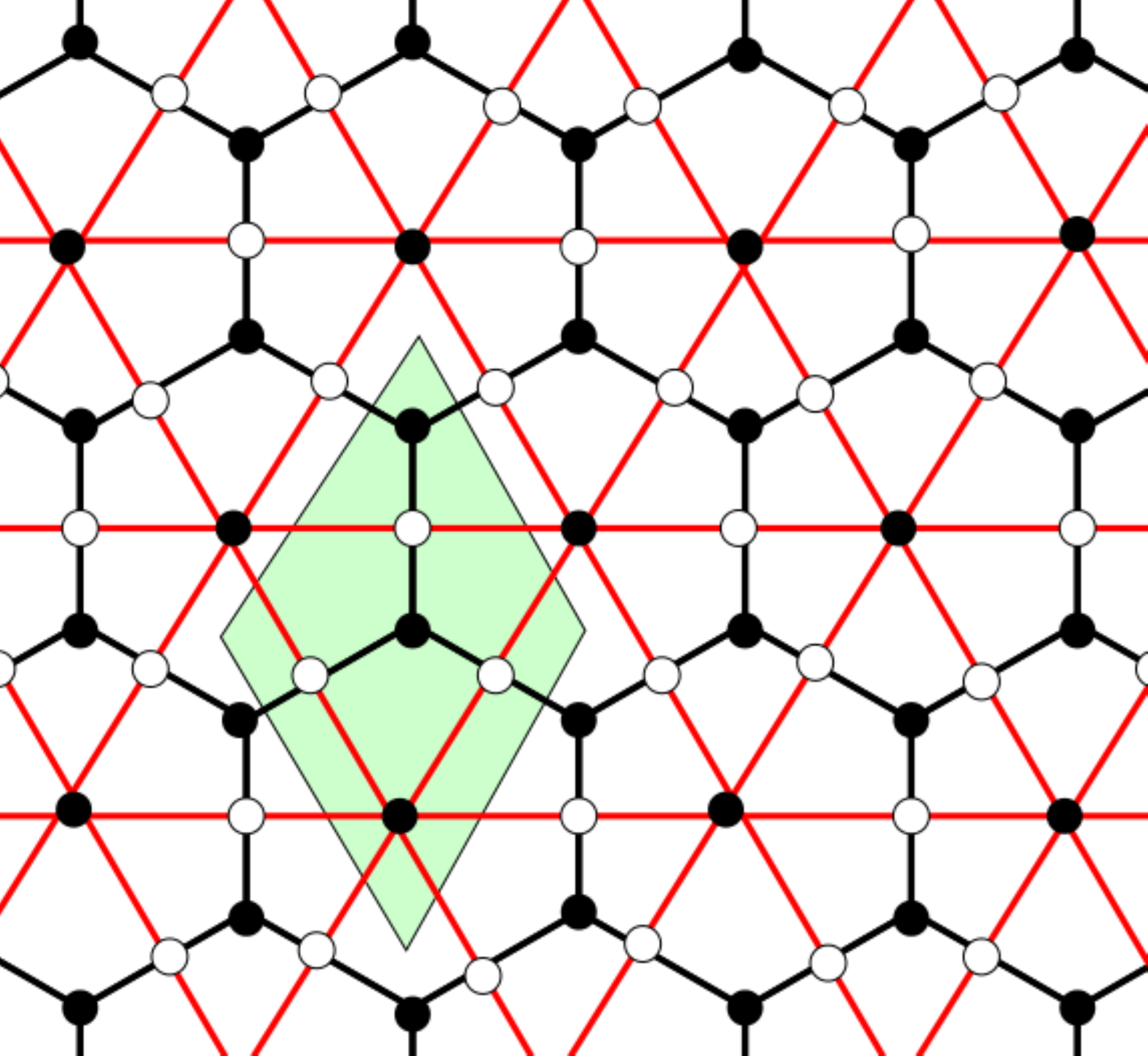} &

\begin{tikzpicture} 
\node at (0,0)
{ \includegraphics[height=2in]{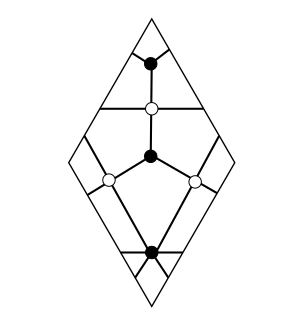}};
\node at (0,1.8) {\tiny{$1$}};
\node at (0.2,1.1) {\tiny{$1'$}};
\node at (0.25,0.25) {\tiny{$2$}};
\node at (-0.4,-0.4) {\tiny{$2'$}};
\node at (0.4,-0.4) {\tiny{$3'$}};
\node at (0,-1.1) {\tiny{$3$}};

\draw[thick,->] (-1.65,0) -- (-.6,1.8);
\node at (-1.65,1) {$\gamma_w$};
\draw[thick,->] (-1.65,0) -- (-.6,-1.8);
\node at (-1.65,-1) {$\gamma_z$};

\node at (-.25,-.65) {$-$};
\node at (-.4,.7) {$-$};
\node at (.7,.25) {$-$};

\end{tikzpicture}
\\
(a) &  (b) &  (c) \\
\end{tabular}
\caption{(a) Diagram of biperiodic triaxial link $\LL$, and fundamental
  domain for $L$.  (b) Overlaid graph $G_{\LL}^b$ and fundamental
  domain for $G_L^b$.  (c) Toroidal graph $G_L^b$, with
  a choice of homology basis, ordered vertices and a choice of
  Kasteleyn signs on edges.}
\label{fig:triaxial}
\end{figure}


\subsection{General Mahler measure theory}

Let $P(x_1,\ldots,x_n) \in \C[x_1^{\pm},\ldots,x_n^{\pm}]$ be non-zero, and 
let $\mathbb{T}^n$ denote the unit torus in $\C^n$. 
The logarithmic {\em Mahler measure} of $P$ is defined by
$$\m(P)=\frac{1}{(2\pi i)^n}\int_{\mathbb{T}^n} 
\log|P(x_1,\ldots,x_n)|\;\frac{dx_1}{x_1}\cdots\frac{dx_n}{x_n}.$$

We now describe the general method for finding the exact Mahler measure of certain two-variable polynomials, which was developed by Boyd and Rodriguez-Villegas \cite{BoydRV2002,brv03}.
See also the discussion leading to \cite[Theorem 2]{Vandervelde}.
Let $P(x,y) \in \C[x,y]$ be a nonzero polynomial of degree $d$ in $y$. 
Let $Y$ be the zero locus of $P(x,y)$ and let $X$ be a smooth projective completion of $Y$.
If we think of $\C[x,y]=\C[x][y]$, then we may write
\[P(x,y)=P^*(x)(y-y_1(x))\cdots (y-y_d(x)),\]
where $y_i(x)$ are algebraic functions of $x$.

By applying Jensen's formula with respect to the variable $y$, to the integral in the definition of Mahler measure, we obtain
\begin{align*}
\m(P(x,y))-\m(P^*(x))= &  \frac{1}{(2 \pi i)^2} \int_{\TT^2}\log |P(x,y)|\frac{dx}{x}\frac{dy}{y}-\m(P^*)\\
=& \frac{1}{(2 \pi i)^2} \int_{\TT^2}\sum_{j=1}^d \log |y-y_j(x)|\frac{dx}{x}\frac{dy}{y}\\
=& \frac{1}{2\pi i}\sum_{j=1}^d  \int_{|x|=1,|y_j(x)|\geq 1} \log|y_j(x)|\frac{dx}{x}\\
=& -\frac{1}{2\pi}  \sum_{j=1}^d \int_{|x|=1,|y_j(x)|\geq 1}  \eta(x,y_j),
\end{align*}
where 
\[\eta(x,y) := \log |x| d \arg y - \log|y| d \arg x\]
is a closed differential form, and $d \arg z= \im(dz/z)$.
We have that 
\[\eta(z,1-z)=dD(z),\]
where $D(z)$ is the Bloch--Wigner dilogarithm given by  
\begin{equation}\label{eq:bw-dilog}
D(z)= \im(\LLi_2(z))+\arg(1-z) \log|z|,
\end{equation}
and
\[\LLi_2(z)=-\int_0^z\frac{\log(1-t)}{t}dt\] 
is the classical dilogarithm. While the value of the classical dilogarithm is dependent on the integration path, 
$D(z)$ is a single-valued continuous function in $\mathbb{P}^1(\C)$ which is real analytic in $\C - \{0,1\}$.    
 
If we can write 
\begin{equation}\label{eq:condition}
x\wedge y_j = \sum_{j_k} \alpha_{j_k} (z_{j_k}\wedge (1-z_{j_k})),
\end{equation}
in $\C(X)^* \wedge \C(X)^*$, 
then we have
\begin{equation}\label{eq:sum}
\m(P(x,y))-\m(P^*(x))= -\frac{1}{2\pi} \sum_{j=1}^d \sum_{j_k} \alpha_{j_k} \left . D(z_{j_k})\right|_{\partial \{|x|=1,|y_j(x)|\geq 1\}}.
\end{equation}

It is not clear a priori that equation~\eqref{eq:condition} can be
solved for any given $P(x,y)$.
Champanerkar~\cite{Cthesis} showed that for the $A$--polynomial of any $1$--cusped hyperbolic $3$--manifold, \eqref{eq:condition} can be solved using Thurston's gluing equations for ideal triangulations.
In addition, if the curve attached to our
polynomial has genus 0, then it can be parametrized (see \cite{Vandervelde}).  In this case, we
will get a solution to \eqref{eq:condition}, possibly with
some extra terms of the form $c\wedge z$, where $c$ is a constant, and
$z$ is a function. Then, we can still reach a closed formula by
integrating $\eta(c,z)$ directly.
Note that $\eta(\omega,z)=0$ when $\omega$ is a root of unity. Thus,
it is more convenient to work in $(\C(X)^*\wedge \C(X)^*)_\Q$, where the subscript 
indicates tensoring by $\Q$, resulting in the torsion elements removed from consideration.

\begin{lemma} \label{lemma} Let $a,b,c,d \in \C$ and $t$ be a variable.  If $ad-bc\not = 0$, we have, in $(\C(t)^*\wedge \C(t)^*)_\Q$,
\[(at+b)\wedge (ct+d)=\frac{act+bc}{ad-bc}\wedge \frac{act+ad}{ad-bc}-(ad-bc)\wedge \frac{act+bc}{act+ad}
-c\wedge (ct+d)-(at+b)\wedge a - c\wedge a.\] 
\end{lemma}
\begin{proof} 
\begin{align*}
\frac{-(act+bc)}{ad-bc}\wedge \frac{act+ad}{ad-bc}=&(-act-bc)\wedge (act+ad)-(-act-bc)\wedge (ad-bc)-(ad-bc) \wedge (act+ad)\\
=&(-act-bc)\wedge (act+ad)+(ad-bc)\wedge \frac{-(act+bc)}{act+ad}
\end{align*}
and
\[(-act-bc)\wedge (act+ad)=(-at-b)\wedge (ct+d)+c\wedge (ct+d)+(-at-b)\wedge a + c\wedge a,\]
and we finally use that $(-x)\wedge y = x \wedge y$.
\end{proof}

\subsubsection*{\bf Properties of the Bloch-Wigner dilogarithm}

We record here some useful properties of the Bloch--Wigner dilogarithm given by \eqref{eq:bw-dilog}. 
A good reference in the subject is Zagier~\cite{Zagier88}.

Its most fundamental property is the five-term relationship
\begin{equation}\label{eq:five-term}
D(x)+D(y)+D(1-xy)+D\left(\frac{1-x}{1-xy}\right)+D\left(\frac{1-y}{1-xy}\right)=0.
\end{equation}
We will often refer to equation~\eqref{eq:five-term} as ``the five-term relation generated by $x$ and $y$.''
In particular,
\begin{equation}\label{eq:inverse}
D\left(\frac{1}{z}\right)=-D(z),{\rm\ and\ }   D(1-z)=-D(z). 
\end{equation}
In addition, we have,
\begin{equation} \label{eq:conjugate}
 D(\overline{z})=-D(z). 
\end{equation}
This identity, which is independent of the five-term relation, implies that $\left . D\right|_\R=0$. 

By taking the five-term relation generated by $z$ and $-z$, we obtain
\begin{equation}\label{eq:square}
2D(z)+2D(-z)=D(z^2).
\end{equation}

Finally, we record a property that expresses $D(z)$ as a combination of dilogarithms evaluated at complex numbers of norm $1$. 
\begin{equation}\label{eq:unitary}
D(z)=\frac{1}{2}\left(D\left(\frac{z}{\overline{z}}\right)+D\left(\frac{1-1/z}{1-1/\overline{z}}\right)
+D\left(\frac{1/(1-z)}{1/(1-\overline{z})}\right)\right).
\end{equation}


\section{Applications}
\label{sec:rth}

\subsection{Proof of the Vol-Det Conjecture for infinite families of links}
In this section, we prove Theorem~\ref{Thm:voldet_ineq}.
We will refer to the notation used in Definition~\ref{def:folner_converge}.

Let $K$ be any hyperbolic alternating link with a reduced alternating diagram, for which the number of bounded $i$--faces of $G(K)$ is $b_i$, for all $i>1$.
By~\cite[Theorem 4.1]{Adams:bipyramids}, we get a volume bound for $K$, which is similar to
equation~\eqref{eq:volbp} for links in $T^2\times I$, by excluding the unbounded face
of the planar link diagram:
$$\vol(K) \leq \volbp(K):= \sum_{f \in \{\substack{\text{bounded}\\ \text{faces of}}
  \ K\}}\vol(B_{|f|}) = \sum\nolimits_i b_i\,\vol(B_i).$$

\begin{theorem}\label{thm:volbp-density}
  Let $\LL$ be any biperiodic alternating link, with toroidally alternating quotient link $L$. 
  Let $K_n$ be alternating hyperbolic links such that \mbox{$\displaystyle K_n\toF \LL$}.  Then
$$ \lim_{n\to\infty}\frac{\volbp(K_n)}{c(K_n)} = \frac{\volbp(L)}{c(L)}.$$
\end{theorem}

\begin{proof} 
  Let $b_{n,i}$ be the number of bounded $i$--faces of $G(K_n)$, for \mbox{$\displaystyle K_n\toF \LL$}.
  Then
$$ \frac{\volbp(K_n)}{c(K_n)} := \frac{\sum_i b_{n,i}\,\vol(B_i)}{c(K_n)}. $$

\begin{figure}
  \centering 
\includegraphics[height=2in]{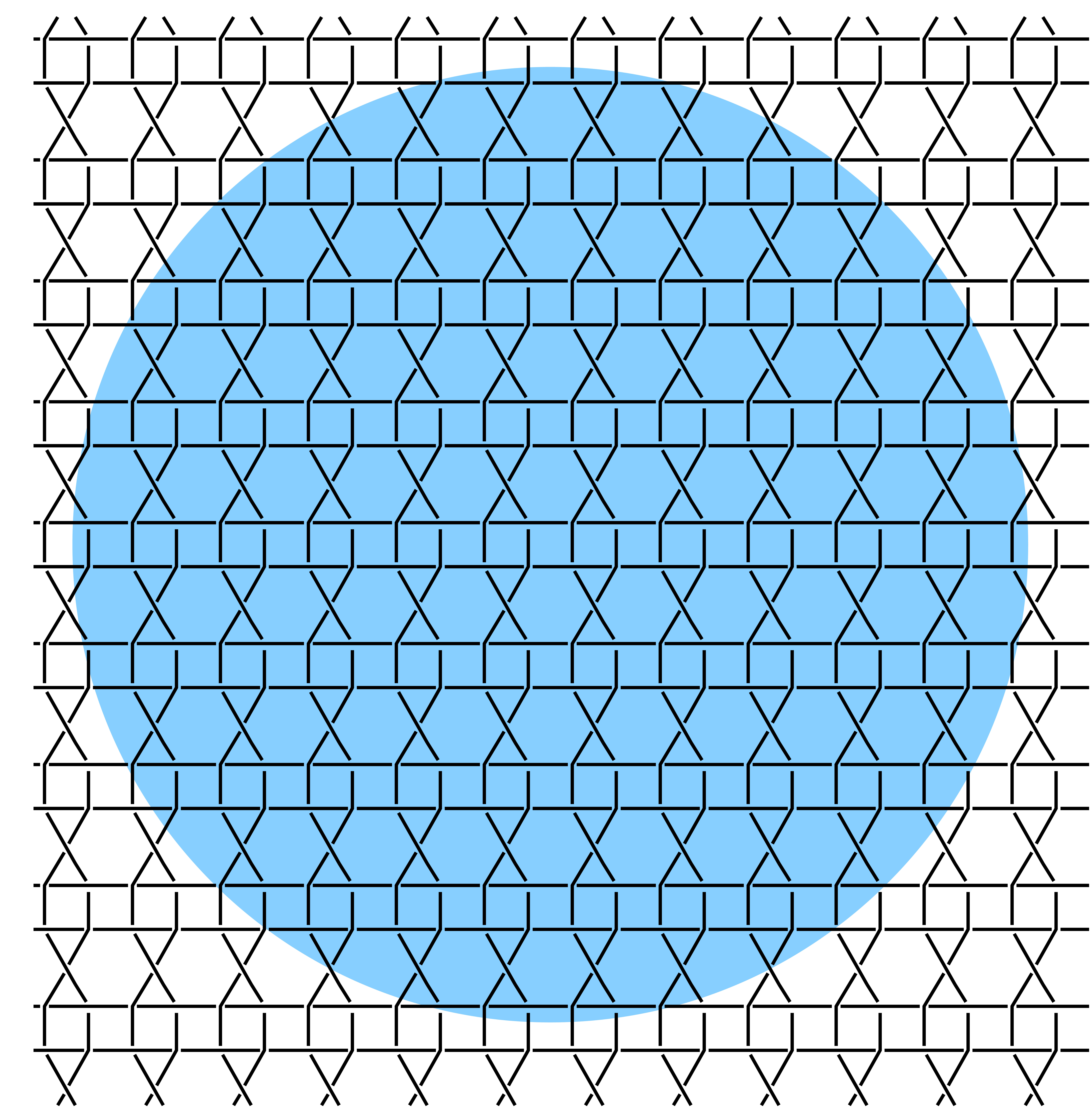}
\caption{Volume bound for F{\o}lner convergence of finite links $K_n$
  to a biperiodic link $\LL$: The part $G_n\subset G(\LL)$ is shown in
  the disc.  As $G_n\subset G(K_n)$, the bipyramid volume density of
  $K_n$ converges to that of $\LL$.}
\label{fig:vol-cvg}
\end{figure}

Let $G_n\subset G(K_n)$ be as in
Definition~\ref{def:folner_converge}.  Since $L=\LL/\Lambda$, the
projection graph $G(L)$ lifts to a $\Lambda$--fundamental domain graph $\G(L)$ for
$G(\LL)$.  We consider three mutually exclusive types of bounded faces of
$G(K_n)$ (see Figure~\ref{fig:vol-cvg}):
\begin{enumerate}
  \item Let $b'_{n,i}$ be the number of $i$--faces of all copies of $\G(L)$ entirely contained in $G_n$. 
  \item Let $b''_{n,i}$ be the number of $i$--faces of $G_n$ that are not counted in (1).  
  \item Let $b'''_{n,i}$ be the number of $i$--faces of $G(K_n)$ which are not in $G_n$.
\end{enumerate}
Note that $b'_{n,i}+b''_{n,i}$ is the number of $i$--faces of $G_n$, and $b'_{n,i}+b''_{n,i}+b'''_{n,i}=b_{n,i}$, the number of bounded $i$--faces of $G(K_n)$.

Now, suppose there are $c_n$ copies of $\G(L)$ entirely contained in $G_n$, so
if $b^L_i$ is the number of $i$--faces of $\G(L)$ then $b'_{n,i} =
c_n\, b^L_i$.
Moreover, we can bound the remaining faces of $G_n$, which are counted in item (2).
Every face of $G_n$ counted in $b''_{n,i}$ is in a copy of $\G(L)$ incident to $\partial G_n$, so that $b''_{n,i} \leq b^L_i\,|\partial G_n|$. 
Thus,
$$ c_n\,|G(L)|\leq |G_n|\leq c_n\,|G(L)|+|G(L)|\,|\partial G_n|.$$
By Definition~\ref{def:folner_converge}, $\displaystyle \frac{|\partial G_n|}{|G_n|}\to 0$, so that
$\displaystyle \frac{c_n\,|G(L)|}{|G_n|}\to 1$ as $n\to\infty$.
Therefore,
$$ \lim_{n\to\infty}\frac{\sum_ib'_{n,i}\,\vol(B_i)}{|G_n|} = \lim_{n\to\infty}\frac{\sum_i c_n\, b^L_i\,\vol(B_i)}{c_n\,|G(L)|} = \frac{\sum_i b^L_i\,\vol(B_i)}{|G(L)|}=\frac{\volbp(L)}{c(L)}. $$

By adding the central axis and stellating each bipyramid, every $B_i$
can be decomposed into $i$ tetrahedra (see~\cite[Figure
  15]{ckp:bal_vol}).  Since each tetrahedron contributes at most
$\vtet$ to the hyperbolic volume, $\vol(B_i)\leq i\,\vtet$.
Therefore, for every copy of $\G(L)$ which is only partially contained
in $G_n$,
$$ \sum\nolimits_i b''_{n,i}\,\vol(B_i) \leq \sum\nolimits_i i\,b^L_i\,\vtet\,|\partial G_n| \leq (4\,|G(L)|)\vtet\,|\partial G_n|.$$
For the last inequality, the $i\,b^L_i$ sum counts with multiplicity the vertices of all faces of $G(L)$, which is $4$--valent, so
the sum over all $i$ is bounded by four times the number of its vertices.

For the bounded $i$--faces of $G(K_n)$ which are not in $G_n$,
$$ \sum\nolimits_i b'''_{n,i}\,\vol(B_i) \leq \sum\nolimits_i i\,b'''_{n,i}\,\vtet \leq 4\vtet\,|G(K_n)-G_n| + 4\vtet\,|\partial G_n|.$$
The last inequality can be seen as follows: the $i\,b'''_{n,i}$ sum counts
with multiplicity the vertices of all bounded faces of $G(K_n)$
that are not in $G_n$.  Since $G(K_n)$ is 4-valent, the sum
over all $i$ is bounded by four times the number of vertices outside
$G_n$ and vertices of $\partial G_n$.

By Definition~\ref{def:folner_converge}, $\displaystyle \frac{|\partial G_n|}{|G_n|}\to 0$ and $\displaystyle \frac{|G_n|}{c(K_n)} \to 1$ as $n\to\infty$.
Thus, $\displaystyle \frac{|G(K_n)-G_n|}{c(K_n)}\to 0$, so 
\begin{eqnarray*}
\lim_{n\to\infty}\frac{\volbp(K_n)}{c(K_n)} &=& \lim_{n\to\infty}\frac{\sum_i (b'_{n,i}+b''_{n,i}+b'''_{n,i})\,\vol(B_i)}{c(K_n)} \\
&=& \lim_{n\to\infty}\frac{\sum_ib'_{n,i}\,\vol(B_i)}{|G_n|} +
{\rm O}\left(\frac{|\partial G_n|}{|G_n|}\right)\vtet +
{\rm O}\left(\frac{|G(K_n)-G_n|}{c(K_n)}\right)\vtet \\
&=& \frac{\volbp(L)}{c(L)}.
\end{eqnarray*}
\end{proof}


\begin{proof}[Proof of Theorem~\ref{Thm:voldet_ineq}]
Since $\vol(K) \leq \volbp(K)$,
$$ \frac{\vol(K_n)}{c(K_n)}\leq \frac{\volbp(K_n)}{c(K_n)}. $$
Hence, the hypothesis $\volbp(L)< 2\pi\,\m(p(z,w))$ and Theorem~\ref{thm:volbp-density} imply that
\begin{equation}\label{eq:rth2}
\lim_{n\to\infty}\frac{\vol(K_n)}{c(K_n)}\leq \lim_{n\to\infty}\frac{\volbp(K_n)}{c(K_n)} = \frac{\volbp(L)}{c(L)} < \frac{2\pi\,\m(p(z,w))}{c(L)}.
\end{equation}

By Theorem~\ref{Thm:det_mp}, \mbox{$\displaystyle K_n\toF \LL$} implies
$$ \lim_{n\to\infty}\frac{\log\det(K_n)}{c(K_n)} = \frac{\m(p(z,w))}{c(L)}.$$
Therefore,
$$ \lim_{n\to\infty}\frac{\vol(K_n)}{c(K_n)} < \lim_{n\to\infty}\frac{2\pi\,\log\det(K_n)}{c(K_n)},$$
which proves the claim. 
\end{proof}

\begin{remark}
The proof above fails without the hypothesis $\volbp(L)< 2\pi\,\m(p(z,w))$, when
$\displaystyle \lim_{n\to \infty} \frac{\vol(K_n)}{c(K_n)}=\lim_{n\to \infty} \frac{2 \pi \log \det(K_n)}{c(K_n)}$.
This happens in only two cases that we know of: the square weave $\W$ and the triaxial link $\LL$ discussed
in Section~\ref{sec:right-angled}.  Nevertheless, we checked
numerically for weaving knots $\displaystyle K_n \toF \W$ (see~\cite{ckp:weaving}) with hundreds of crossings that the Vol-Det Conjecture does hold.
\end{remark}

\subsection{Bound on volume change under augmentation}

In \cite{ckp:density}, it was shown that the Vol-Det Conjecture implies the
following conjecture, which would be a new upper bound for how much
the volume can change after drilling out an augmented unknot:
\begin{conjecture}[\cite{ckp:density}]\label{conj:voldet2}
For any hyperbolic alternating link $K$ with an augmented unknot $B$
around any two parallel strands of $K$, 
\[ \vol(K) < \vol(K\cup B) \leq 2\pi\log\det(K). \]
\end{conjecture}

In this section, we prove Conjecture~\ref{conj:voldet2} for infinite families of knots or links that include
almost all $K_n$ for every sequence \mbox{$\displaystyle K_n\toF \LL$} as in Theorem~\ref{Thm:voldet_ineq}.

\begin{corollary}\label{cor:augmented}
Let $K_n \toF \LL$ be links satisfying the conditions of Theorem~\ref{Thm:voldet_ineq}.
Then for almost all $n$,
\[ \vol(K_n) < \vol(K_n\cup B) < 2\pi\log\det(K_n). \]
\end{corollary}

\begin{proof}
Since volume increases under Dehn drilling, $\vol(K_n) < \vol(K_n\cup B)$.  
Although we do not know that volume densities of $K_n$ converge to that of $\LL$ (see \cite[Conjecture 6.5]{ckp:bal_vol}), 
Theorem~\ref{thm:volbp-density}
implies $\displaystyle \limsup_{n\to \infty} \frac{\vol(K_n)}{c(K_n)} \leq \frac{\volbp(L)}{c(L)}$.

\begin{figure}
  \centering 
\hspace*{0.7in} \includegraphics[scale=0.25]{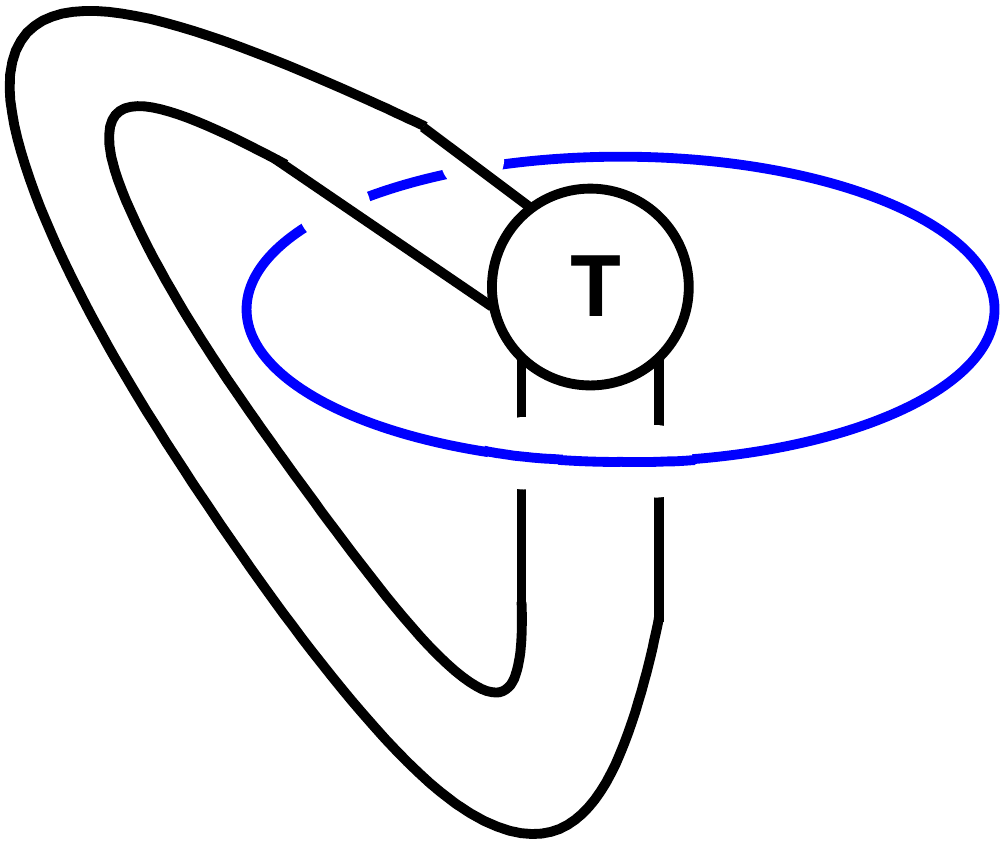}\qquad
\hspace*{0.5in} \includegraphics[scale=0.25]{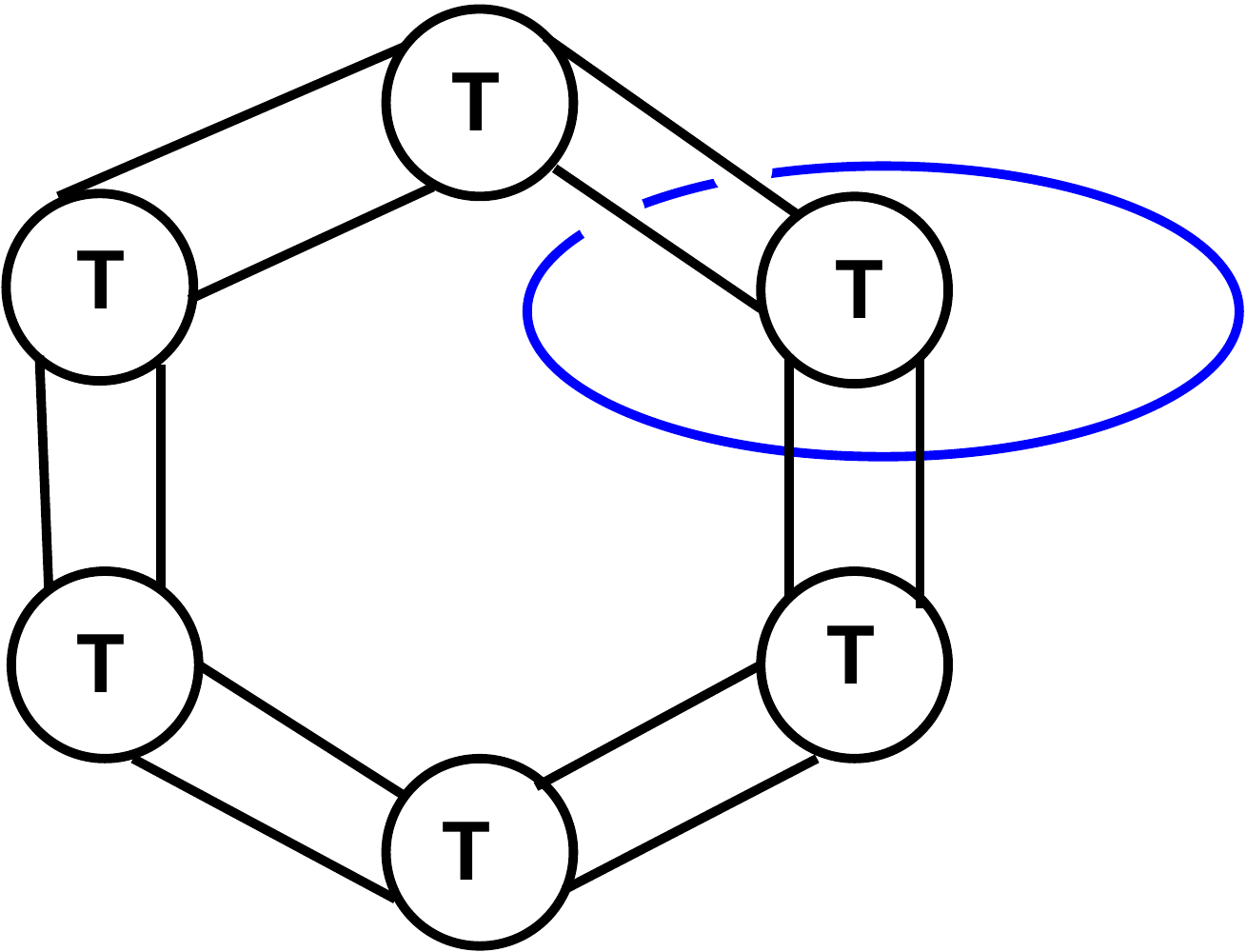}
  \caption{Augmented hyperbolic alternating links $K\cup B$ and $K^m\cup B$}
  \label{fig:augmented}
\end{figure}

Let $K \cup B$ be any augmented alternating link, and let $K^m$ denote
the $m$--periodic alternating link with quotient $K$, formed by taking
$m$ copies of a tangle $T$ as in Figure~\ref{fig:augmented}.  It was shown in \cite{ckp:density} that 
$$ \lim_{m\to\infty}\frac{\vol(K^m)}{c(K^m)}= \frac{ \vol(K\cup B)}{c(K)}, 
\text{\ and \ }
 \lim_{m\to\infty}\frac{2\pi\log\det(K^m)}{c(K^m)} = \frac{2\pi\log\det(K)}{c(K)}.$$

Thus, for all $m$, $\displaystyle\limsup_{n\to \infty} \frac{\vol(K_n^m)}{c(K_n^m)} \leq \frac{\volbp(L)}{c(L)} + \epsilon(m)$, such that $\lim_{m\to\infty}\epsilon(m)=0$.
It follows that $\displaystyle \limsup_{m,n\to \infty} \frac{\vol(K_n^m)}{c(K_n^m)} \leq \frac{\volbp(L)}{c(L)}$.
Therefore,
$$ \frac{ \vol(K_n\cup B)}{c(K_n)}=\lim_{m\to \infty} \frac{\vol(K_n^m)}{c(K_n^m)} \leq \limsup_{m,n\to \infty} \frac{\vol(K_n^m)}{c(K_n^m)} \leq \frac{\volbp(L)}{c(L)} < \frac{2\pi\log\det(K_n)}{c(K_n)}$$ 
for almost all $n$, where the final inequality follows by inequality~\eqref{eq:rth2} and Theorem~\ref{Thm:det_mp}.
\end{proof}


\section{Proven examples for Conjecture \ref{conj:volbipy}}
\label{sec:examples}

To review notation, recall that $\vtet$ is the volume of the regular
ideal tetrahedron, and $\voct$ is the volume of the regular ideal octahedron:
\[\vtet=D\left(e^{i\pi/3}\right)=D\left(\frac{1+\sqrt{3}i}{2}\right) \approx 1.01494, \quad \voct= 4 D\left(e^{i\pi/2}\right)=4D(i)\approx 3.66386.\]

\subsection{Square weave}
Our first example is the square weave $\W$, as shown in Figure \ref{fig-square}, which was discussed in Example 1 of Section~\ref{sec:background}.
Let $W$ be its alternating quotient link in $T^2 \times I$ as in Section~\ref{sec:background}.
By equation~(\ref{eqn:charpol-weave}),
\[ p_{W}(z,w)=-\left(4+w+\frac{1}{w}+z+\frac{1}{z}\right). \]

In \cite{Boyd:special}, Boyd gives the main idea how to prove a
formula for the Mahler measure of $p_{W}(z,w)$. Below we provide the
missing details, including the dilogarithm evaluation using formula~\eqref{eq:sum}.

\begin{theorem}
\label{thm:mm-weave}
$$ 2\pi\, \m(p_{W}) = 8 D(i) = 2 \voct .$$
Consequently, $\vol((T^2\times I) - W) = \volbp(W) = 2\pi \,\m(p_{W})$.
\end{theorem}

\begin{proof}
Consider the factorization due to Boyd \cite{Boyd:special}:
\begin{align*}
q(z,w)=-p_{W}(z/w,wz)=&4+\left(wz+\frac{1}{wz}+\frac{z}{w}+\frac{w}{z}\right)\\
=&\frac{1}{wz}(1+iw+iz+wz)(1-iw-iz+wz).
\end{align*}
Note that $\m(1-iw-iz+wz)=\m(1+iw+iz+wz)$ since one is obtained from
the other by $z \to -z$ and $w \to -w$,
which does not alter the Mahler measure. Hence $\m(p_{W})=2\m(q_1)$ where
$q_1(z,w)=1+iw+iz+wz$.

Let us compute  $\m(q_1)$. Setting
$w=e^{i\theta}$ we get 
\[|z|=\left|\frac{1+iw}{w+i}\right|=\left|\frac{1+iw}{1-iw}\right|=\left|\frac{1+e^{i(\theta+\pi/2)}}{1-e^{i(\theta+\pi/2)}}\right|=\left|\cot\left(\frac{2\theta+\pi}{4}\right)\right|.\]
Then $|z|\geq 1$ iff $-\pi\leq \theta \leq 0$. Therefore we have to integrate between $w=-1$ and $w=1$.
The wedge product can be decomposed as 
$$ w\wedge z = w\wedge \frac{1+iw}{i+w}=w \wedge \frac{1+iw}{1-iw} = iw \wedge (1+iw) -iw\wedge (1-iw).$$
Applying~\eqref{eq:sum}, we evaluate $-\frac{1}{2\pi}(D(-iw)-D(iw))$
on the boundary $w\rvert_{\text{\tiny -1}}^{\text{\tiny 1}}$ to obtain
\[2\pi \m(q_1) =-D(-i\cdot 1)+D(i\cdot 1)+D(-i\cdot (-1))-D(i\cdot (-1))=4 D(i).\]

Thus, we obtain the first claim:
\[2\pi \m(p_{W})= 4\pi \m(q_1)=8 D(i)=2\voct.\]

By \cite[Theorem 3.5]{ckp:bal_vol},
$$ \vol((T^2\times I) -W) = \volbp(W) = 2\voct.$$
Thus, Conjecture~\ref{conj:volbipy} is verified for the square weave $\W$ with an equality.
\end{proof}

\subsection{Triaxial link } \label{sec:right-angled}
Next, we consider the triaxial link $\LL$ as shown in Figure~\ref{fig:triaxial}, which was discussed in Example 2 of Section~\ref{sec:background}.
Let $L$ be its alternating quotient link in $T^2 \times I$ as in Section~\ref{sec:background}.
By equation~(\ref{eqn:charpol-triaxial}),
\[ p_{L}(z,w)=6-\left(w+\frac{1}{w}+z+\frac{1}{z} +\frac{w}{z}+\frac{z}{w}\right). \]

In \cite{Boyd:special}, Boyd mentions without giving the proof that
the Mahler measure of $p_{L}(z,w)$ can be found by using equation
\eqref{eq:sum}.  Below we provide the proof.

\begin{theorem}
\label{thm:mm-triaxial}
$$ 2\pi\, \m(p_{L}) = 10 D\left(\frac{1+\sqrt{3}}{2}\right)  = 10 \vtet .$$
Consequently, $\vol((T^2\times I) - L) = \volbp(L) = 2\pi \,\m(p_{L})$.
\end{theorem}

\begin{proof}
We can parametrize the curve defined by $p_{L}(z,w)=0$ by using standard algorithms (see, e.g., \cite[Chapter 4]{SWPD}). We obtain
\[ z=-\frac{(2t-1)(t-1)}{t + 1},\qquad w=-\frac{(t-2)(t-1)}{t(t+1)}.\]

Setting $w=e^{i\theta}$ we write
\[e^{i\theta}=-\frac{(t-2)(t-1)}{t(t+1)} \implies
 (e^{i\theta}+1)t^2+(e^{i \theta}-3)t +2=0.\] Since
$\displaystyle{|z|=\left|\frac{(2t-1)(t-1)}{t + 1}\right|}$ and we have to integrate for $|z|\geq1$,
it can be seen that the integration domain is given by $\theta \in (0,2\pi)$ and that this corresponds to a path for $t$ that 
has boundary points in $t=\frac{1-\sqrt{3} i}{2}$ and $t=\frac{1+\sqrt{3}i}{2}$. 

We also have
\begin{align*}
w\wedge z = &\frac{(t-2)(t-1)}{t(t+1)} \wedge \frac{(2t-1)(t-1)}{t + 1}\\
=&(t-2)\wedge (2t-1)+(t-2)\wedge (t-1)- (t-2)\wedge (t+1)+(t-1)\wedge (2t-1)\\
&-t \wedge (2t-1)-t \wedge (t-1)+t \wedge (t+1)-(t+1)\wedge (2t-1).\\
\end{align*}

Applying Lemma~\ref{lemma} (and ignoring the terms of the form $(\pm 1)\wedge x$ and $x \wedge (\pm 1)$) we can express every term as a combination of terms of the form $\alpha \wedge 
(1-\alpha)$ as follows:
\begin{align*}
(t-2)\wedge (2t-1)=&\frac{4-2t}{3}\wedge \frac{2t-1}{3}-3\wedge \frac{4-2t}{2t-1}-2\wedge (2t-1)\\
(t-2)\wedge (t-1)=&(2-t)\wedge (t-1)\\
(t-2)\wedge (t+1)=&\frac{2-t}{3}\wedge \frac{t+1}{3}-3\wedge \frac{2-t}{t+1}\\
(t-1)\wedge (2t-1)=&(2-2t)\wedge (2t-1) -2\wedge (2t-1)\\
t \wedge (2t-1)=&(2t)\wedge (1-2t)-2\wedge(2t-1)\\
t \wedge (t-1)=&t\wedge (1-t)\\
t \wedge (t+1)=&(-t)\wedge (t+1)\\
(t+1)\wedge (2t-1)=&\frac{2t+2}{3}\wedge \frac{1-2t}{3}-3\wedge \frac{2t+2}{2t-1}-2\wedge (2t-1).\\
\end{align*}

Thus, we obtain
\begin{align*}
w \wedge z =&  \frac{4-2t}{3}\wedge \frac{2t-1}{3}
+(2-t)\wedge (t-1)
-\frac{2-t}{3}\wedge \frac{t+1}{3}+(2-2t)\wedge (2t-1)\\&-(2t)\wedge (1-2t)
-t \wedge (1-t)+(-t) \wedge (t+1)-\frac{2t+2}{3}\wedge \frac{1-2t}{3}.\\
\end{align*}
Using equation~\eqref{eq:sum}, this integrates to
\begin{align*}
 &-D\left(\frac{2t-1}{3}\right)-D(t-1)+D\left(\frac{t+1}{3}\right)-D(2t-1)
-D(2t)-D(t)+D(-t)-D\left(\frac{2t+2}{3}\right)\\
=&-D\left(\frac{2t-1}{3}\right)+D\left(\frac{1-2t}{3}\right) -D(t-1)+D(1-t)-D(2t-1)+D(1-2t) \\
&+D\left(\frac{t+1}{3}\right)+D(-t).
\end{align*}
In order to integrate we must evaluate the formula above in
$\frac{1-\sqrt{3}i}{2}$ and $\frac{1+\sqrt{3}i}{2}$ and take the
difference. But this is the same as evaluating in
$\frac{1+\sqrt{3}i}{2}$ and multiplying by 2, since
$D(z)=-D(\overline{z})$.
Using the formulas in equation~\eqref{eq:inverse} and \eqref{eq:conjugate}, 
we obtain further simplifications:
\begin{align*}
2\pi\m(p_{L})=&2D\left(\frac{i}{\sqrt{3}}\right)-2D\left(\frac{-i}{\sqrt{3}}\right)
+2D\left(\frac{-1+\sqrt{3}i}{2}\right)-2D\left(\frac{1-\sqrt{3}i}{2}\right)
+2D(\sqrt{3}i)\\
&-2D(-\sqrt{3}i) -2D\left(\frac{3+\sqrt{3}i}{6}\right)-2D\left(-\frac{1+\sqrt{3}i}{2}\right)\\
=&
4D\left(\frac{-1+\sqrt{3}i}{2}\right)-2D\left(\frac{1-\sqrt{3}i}{2}\right)
+8D(\sqrt{3}i)-2D\left(\frac{3+\sqrt{3}i}{6}\right).\\
\end{align*}

Using the identity \eqref{eq:square}, we have
\[2D\left(\frac{-1+\sqrt{3}i}{2}\right)+2D\left(\frac{1-\sqrt{3}i}{2}\right)=D\left(-\frac{1+\sqrt{3}i}{2}\right)=-D\left(\frac{-1+\sqrt{3}i}{2}\right)\]
and
\[3D\left(\frac{-1+\sqrt{3}i}{2}\right)+2D\left(\frac{1-\sqrt{3}i}{2}\right)=0.\]

The five-term relation \eqref{eq:five-term} generated by $\frac{1}{\sqrt{3}i}$ and $\frac{1+\sqrt{3}i}{2}$ leads to the following identity.
\[-2D(\sqrt{3}i)-D\left(\frac{1-\sqrt{3}i}{2}\right)+D\left(\frac{3+\sqrt{3}i}{6}\right)=0.\]
The five-term relation \eqref{eq:five-term} generated by $1+\sqrt{3}i$ and $\frac{-1+\sqrt{3}i}{2}$ yields
\[D\left(\frac{-1+\sqrt{3}i}{2}\right)-D\left(\frac{3+\sqrt{3}i}{6}\right)=0.\]

Putting all of this together, we get
\begin{align*}
2\pi\m(p_{L})=& 4D\left(\frac{-1+\sqrt{3}i}{2}\right)-2D\left(\frac{1-\sqrt{3}i}{2}\right)
+8D(\sqrt{3}i)-2D\left(\frac{3+\sqrt{3}i}{6}\right)\\
=&-10D\left(\frac{1-\sqrt{3}i}{2}\right)=10D\left(\frac{1+\sqrt{3}i}{2}\right)=
10\vtet.
\end{align*}
By \cite[Theorem 3.5]{ckp:bal_vol},
$$\vol((T^2\times I) -L) = \volbp(L) = 10\vtet.$$ 
Thus, Conjecture \ref{conj:volbipy} is verified for the triaxial link $\LL$ with an equality.
\end{proof}

\subsection{Rhombitrihexagonal link}

\begin{center}
\begin{figure}[h]
\begin{tabular}{ccc}
 \includegraphics[height= 2 in]{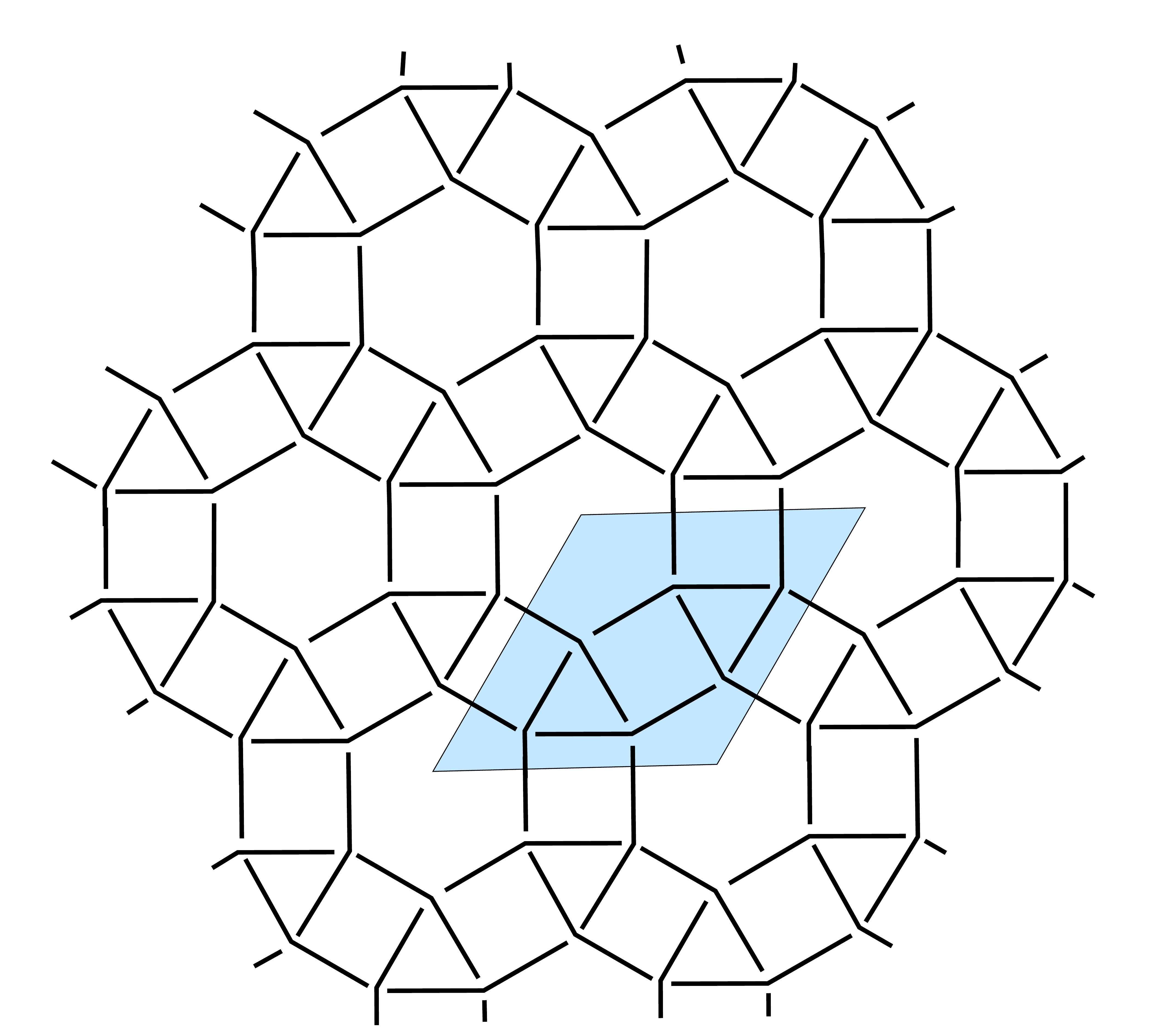} &
 \includegraphics[height=1.75 in]{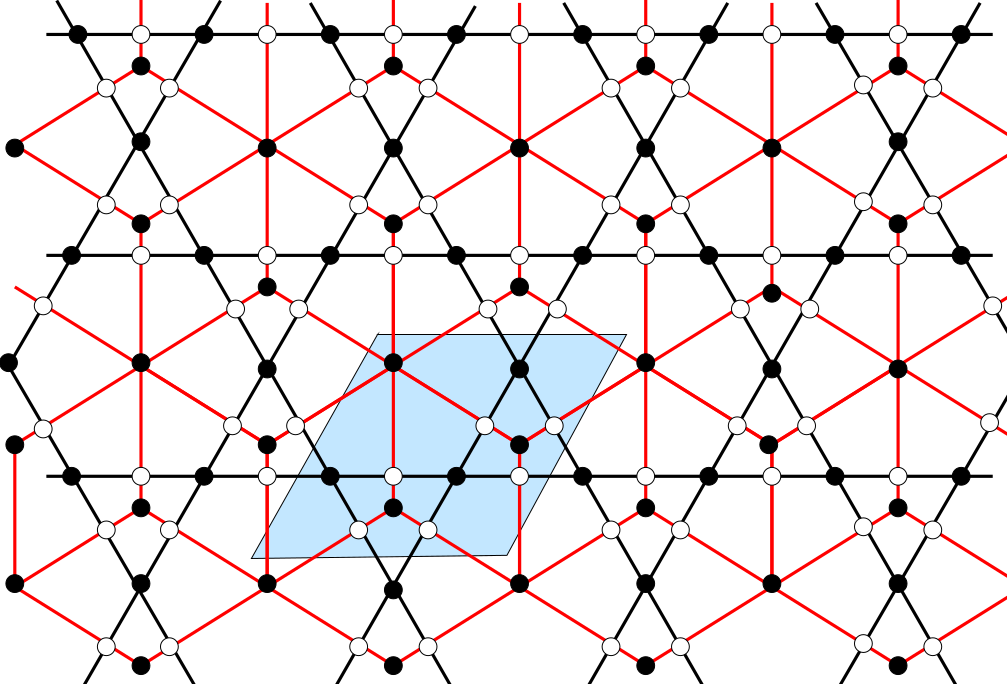}  \\
(a) &  (b)
\end{tabular} 
\caption{(a) Diagram of the biperiodic Rhombitrihexagonal link $\RR$, and
  fundamental domain for $R$.  (b) Overlaid graph $G_{\RR}^b$ and
  fundamental domain for $G_{R}^b$.}
\label{fig:rthGb}
\end{figure}
\end{center}

Figure~\ref{fig:rthGb} shows the rhombitrihexagonal link $\RR$ and its alternating quotient link $R$ in
$T^2 \times I$.  For the fundamental domain for $G_{\RR}^b$ as in Figure~\ref{fig:rthGb} (middle), $p(z,w)= \det\K(z,w)$ is 
$$ p_{R}(z,w) = 6\,(6 -1/w - w - 1/z - z - w/ z - z/w ). $$

\begin{corollary}
  \label{cor:mm-rth}
  $$ 2\pi\, \m(p_{R}) = 2\pi\log (6)+ 10 \vtet.$$
Consequently, $\vol((T^2\times I) - R) = \volbp(R) < 2\pi \,\m(p_{R})$.
\end{corollary}
\begin{proof}

Using Theorem \ref{thm:mm-triaxial}, we see that $2\pi \m(p_{R}) = 2\pi\log (6)+ 10 \vtet .$
 
By \cite[Theorem~3.5]{ckp:bal_vol},
$\volbp(R) = \vol((T^2 \times I) -R) = 10\vtet + 3\voct$.

Hence,
\begin{eqnarray*}
\volbp(R) = \vol((T^2 \times I) -R) & = & 10\, \vtet + 3\,\voct  \approx 21.141 \\ 
& < &  10\, \vtet + 2\pi \log(6)  \approx 21.407 \\ 
&=& 2 \pi\ \m(p_{R}(z,w)). 
\end{eqnarray*}
\end{proof}

\subsection{The link $\CC_0$}

\begin{center}
\begin{figure}[h]
\begin{tabular}{cc}
 \includegraphics[height=1.75 in]{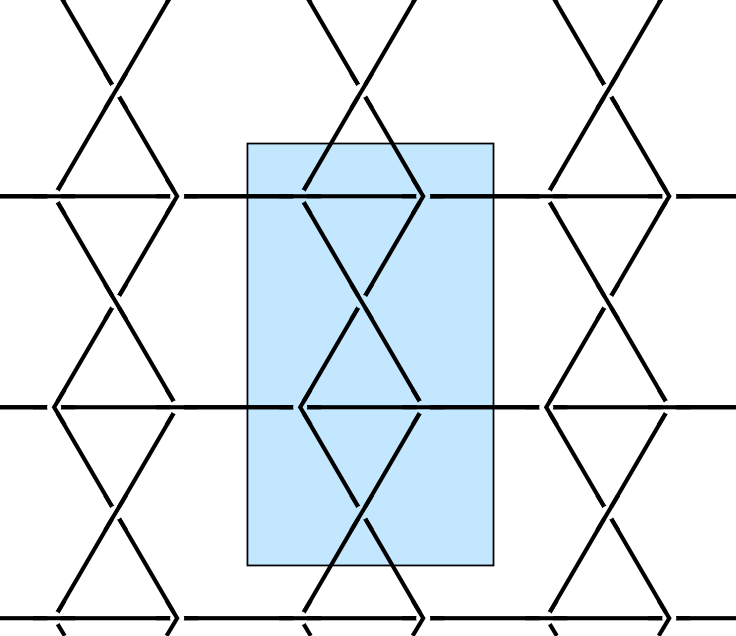} &
\hspace*{2cm}
 \includegraphics[height=1.75 in]{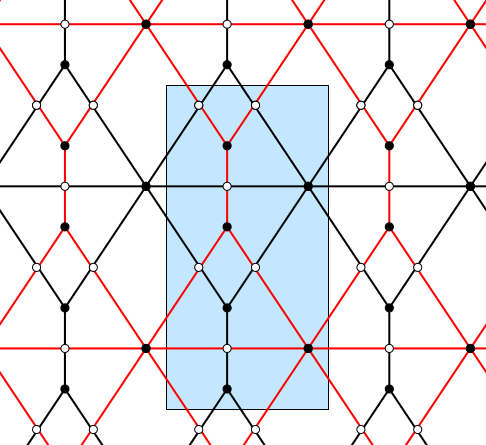}  \\
(a) & \hspace*{2cm} (b)
\end{tabular} 
\caption{(a) Diagram of biperiodic link $\CC_0$, and fundamental
  domain for $C_0$.  (b) Overlaid graph $G_{\CC_0}^b$ and fundamental
  domain for $G_{C_0}^b$.}
\label{fig:c0}
\end{figure}
\end{center}

Figure~\ref{fig:c0}(a) shows the biperiodic alternating link $\CC_0$,
and fundamental domain for its alternating quotient link $C_0$ in $T^2
\times I$.  For the fundamental domain for $G_{\CC_0}^b$ as in
Figure~\ref{fig:c0}(b), we have
\[p_{C_0}(z,w)=(-z(w^2-4w+1)+w^2+4w+1)^2.\]
\begin{theorem}
\label{thm:mm-c0}
 $$ 2\pi\, \m(p_{C_0}) = 16D\left((2+\sqrt{3}) i\right)+
  \frac{8\pi}{3}\log(2+\sqrt{3}). $$
Consequently, $\displaystyle \vol((T^2\times I) - C_0) = \volbp(C_0) < 2\pi \,\m(p_{C_0})$.
\end{theorem}

\begin{proof}
Let $q(z,w)=-z(w^2-4w+1)+w^2+4w+1$.
Since $p_{C_0}(z,w)= q(z,w)^2$, it's enough to compute $\m(q)$. 

  In the equation $q(z,w)=0$ 
solve for $z$ in terms of $w$. Setting
  $w=e^{i\theta}$ we get
\[|z|=\left| \frac{w+4+w^{-1}}{w-4+w^{-1}}\right|=\left| \frac{\cos \theta +2}{\cos \theta -2}\right|.\]
 Hence $|z|\geq 1$ iff $\cos \theta\geq 0$ iff $-\frac{\pi}{2}\leq \theta \leq \frac{\pi}{2}$.

The wedge product leads to
\begin{align*}
w\wedge z =& w \wedge \frac{w^2+4w+1}{w^2-4w+1}\\
=& w \wedge \frac{(1+(2+\sqrt{3})w)(1+(2-\sqrt{3})w)}{(1-(2+\sqrt{3})w)(1-(2-\sqrt{3})w)}\\
=& w \wedge (1+(2+\sqrt{3})w) + w\wedge (1+(2-\sqrt{3})w) - w \wedge (1-(2+\sqrt{3})w) \\
& - w \wedge (1-(2-\sqrt{3})w)\\
= & (2+\sqrt{3}) w \wedge (1+(2+\sqrt{3})w) -(2+\sqrt{3}) \wedge (1+(2+\sqrt{3})w)\\ 
 & + (2-\sqrt{3}) w \wedge (1+(2-\sqrt{3})w) -(2-\sqrt{3}) \wedge (1+(2-\sqrt{3})w)\\ 
 & -(2+\sqrt{3}) w \wedge (1-(2+\sqrt{3})w)+(2+\sqrt{3}) \wedge (1-(2+\sqrt{3})w)\\
& -(2-\sqrt{3}) w \wedge (1-(2-\sqrt{3})w) +(2-\sqrt{3}) \wedge (1-(2-\sqrt{3})w).\\ 
\end{align*}

The Mahler measure of the leading coefficient polynomial equals
\[\m(w^2-4w+1) = \m\left((w-(2+\sqrt{3}))(w-(2-\sqrt{3}))\right)=\log(2+\sqrt{3}).\]

By applying equation~\eqref{eq:sum}, this gives
\begin{align*}
2\pi \, \m(q) -2\pi \log(2+\sqrt{3})= & -2D\left(-(2+\sqrt{3}) i\right)+ 2D\left((2+\sqrt{3}) i\right)-2D\left(-(2-\sqrt{3}) i\right) \\
&+ 2D\left((2-\sqrt{3}) i\right)
 +\log (2+\sqrt{3})\int_{-i}^i d \arg \left(\frac{1+(2+\sqrt{3})w}{1-(2+\sqrt{3})w}\right)\\
&+\log(2-\sqrt{3})\int_{-i}^i d \arg \left(\frac{1+(2-\sqrt{3})w}{1-(2-\sqrt{3})w}\right).
\end{align*}

\begin{lemma}
\label{lem:R-integral} We have
\begin{equation}\label{eq:R-integral}
\int_{-i}^i d \arg \left(\frac{1+Rw}{1-Rw}\right)
=2\arctan\left(\frac{2}{R-R^{-1}}\right).
\end{equation}
\end{lemma}
\begin{proof}
\begin{align*}
I(R):=\int_{-i}^i d \arg \left(\frac{1+Rw}{1-Rw}\right)=& \int_{-i}^{i} \im \left(\frac{R dw}{1+Rw}+\frac{R dw}{1-Rw}\right) \\
=& -2\int_{-\pi/2}^{\pi/2}\re \left(\frac{d \theta}{R^{-1}e^{-i\theta}-Re^{i\theta}}\right)\\
=&2(R-R^{-1})\int_{-\pi/2}^{\pi/2}\frac{\cos \theta  d\theta}{R^2+R^{-2}-2 \cos(2\theta)}\\
=&2(R-R^{-1}) \int_{-1}^1\frac{ds}{(R-R^{-1})^2+4s^2},
\end{align*}
where we have set $s=\sin\theta$. Therefore, the lemma follows. 
\end{proof}

By specializing in $R=2+\sqrt{3}$, we obtain 
\[I(2+\sqrt{3})=\frac{\pi}{3}=-I(2-\sqrt{3}).\]
By using properties \eqref{eq:inverse} and \eqref{eq:conjugate}  we finally get
\begin{align*}
2\pi \,\m(q)= & 8D\left((2+\sqrt{3}) i\right)+ 2\pi \log(2+\sqrt{3})-\frac{\pi}{3}\log(2+\sqrt{3})+\frac{\pi}{3}\log(2-\sqrt{3})\\
=& 8D\left((2+\sqrt{3}) i\right)+ \frac{4\pi}{3}\log(2+\sqrt{3}) \approx 10.40161017.\\
\end{align*}

Therefore, $2\pi \,\m(p_{C_0})=4\pi\,\m(q) \approx 20.80322034.$
By \cite[Theorem 3.5]{ckp:bal_vol},
$$\vol((T^2\times I) -C_0) = \volbp(C_0) = 20\vtet \approx 20.29883212$$
Thus, Conjecture~\ref{conj:volbipy} is verified for the link $\CC_0$.
\end{proof}

\subsection{The link $\CC_1$}
\begin{center}
\begin{figure}[h]
\begin{tabular}{cc}
 \includegraphics[height=1.75 in]{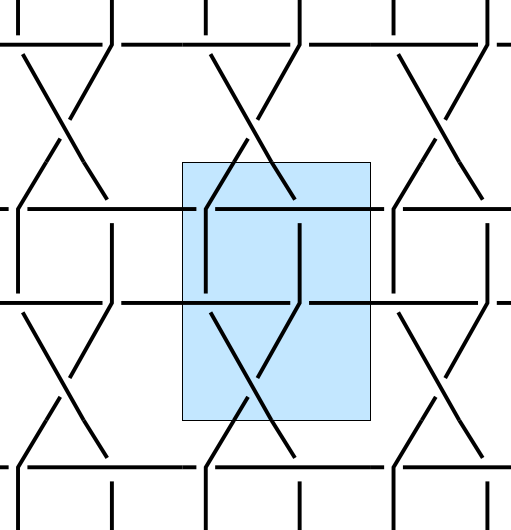} &
\hspace*{2cm}
 \includegraphics[height=1.75 in]{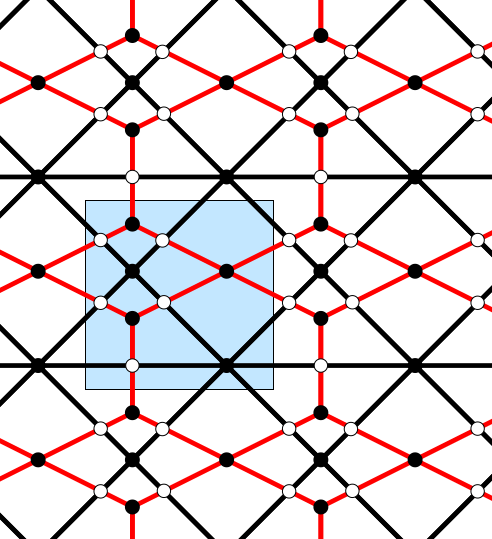}  \\
(a) & \hspace*{2cm} (b)
\end{tabular} 
\caption{(a) Diagram of biperiodic link $\CC_1$, and fundamental
  domain for $C_1$.  (b) Overlaid graph $G_{\CC_1}^b$ and fundamental
  domain for $G_{C_1}^b$.}
\label{fig:c1}
\end{figure}
\end{center}

Figure~\ref{fig:c1}(a) shows the biperiodic alternating link $\CC_1$,
and fundamental domain for its alternating quotient link $C_1$ in $T^2
\times I$.  For the fundamental domain for $G_{\CC_1}^b$ as in
Figure~\ref{fig:c1}(b), we have
\[p_{C_1}(z,w)= (1+w^2)(1-z)^2-w(6+20z+6z^2).\]

\begin{theorem}
\label{thm:mm-c1}
  $$ 2\pi\, \m(p_{C_1}) = 16D\left((1+\sqrt{2}) i\right)+2\pi \log(1+\sqrt{2}). $$
Consequently, $\vol((T^2\times I) - C_1) = \volbp(C_1) < 2\pi \,\m(p_{C_1})$.
\end{theorem}

\begin{proof}
Let
\[ p_1(z,w)=p_{C_1}(z,w^2)=
(w^4-6w^2+1)z^2-2(w^4+10w^2+1)z+(w^4-6w^2+1). \]
Then $\m(p_{C_1})=\m(p_1)$. Solving for $z$ in terms of $w$, we get two roots

\[z_\pm = \frac{w^4+10w^2+1 \pm 4\sqrt{2} w(w^2+1)}{w^4-6w^2+1}.\]

We need to impose conditions for $|z_\pm|\geq 1$. Set $w=e^{i\theta}$. Then
\begin{align*}
z_\pm  =&\frac{w^2+10+w^{-2} \pm 4\sqrt{2} (w+w^{-1})}{w^2-6+w^{-2}}
= \frac{\cos(2\theta)+5\pm 4 \sqrt{2}\cos \theta}{\cos(2\theta)-3}\\
=&\frac{\cos^2\theta\pm 2 \sqrt{2}\cos \theta +2}{\cos^2\theta -2}
= \frac{\cos \theta\pm \sqrt{2}}{\cos \theta\mp \sqrt{2}}.
\end{align*}

Thus, we get $|z_+|\geq 1$ iff $\cos \theta\geq 0$ and $|z_-|\geq 1$ iff $\cos \theta\leq 0$.

\begin{align*}
w\wedge z_\pm =& w \wedge \frac{(1+(1\pm \sqrt{2})w)^2(1-(1\mp \sqrt{2})w)^2}{(1+(1+\sqrt{2})w)(1-(1-\sqrt{2})w)(1+(1-\sqrt{2})w)(1-(1+\sqrt{2})w)}\\
=&  w \wedge \frac{(1+(1\pm \sqrt{2})w)(1-(1\mp \sqrt{2})w)}{(1-(1\pm \sqrt{2})w)(1+(1\mp \sqrt{2})w)}\\
=& w \wedge (1+(1\pm \sqrt{2})w)+w \wedge (1-(1\mp \sqrt{2})w) -w \wedge (1-(1\pm \sqrt{2})w)\\
&-w \wedge (1+(1\mp \sqrt{2})w)\\
=& (1\pm \sqrt{2})w \wedge (1+(1\pm \sqrt{2})w)-(1\pm \sqrt{2})\wedge  (1+(1\pm \sqrt{2})w)\\
&+(1\mp \sqrt{2})w \wedge (1-(1\mp \sqrt{2})w)-(1\mp \sqrt{2})\wedge  (1-(1\mp \sqrt{2})w)\\
&-(1\pm \sqrt{2})w \wedge (1-(1\pm \sqrt{2})w)+(1\pm \sqrt{2})\wedge  (1-(1\pm \sqrt{2})w)\\
&- (1\mp \sqrt{2})w \wedge (1+(1\mp \sqrt{2})w)+(1\mp \sqrt{2})\wedge  (1+(1\mp \sqrt{2})w).\\
\end{align*}

The Mahler measure of the leading coefficient polynomial equals
\begin{align*}
\m(w^4-6w^2+1) =& \m \left((w+(1+\sqrt{2}))(w-(1-\sqrt{2}))(w+(1-\sqrt{2}))(w-(1+\sqrt{2}))\right)\\
= & 2\log(1+\sqrt{2}).
\end{align*}

By putting together the cases of $z_+$ and $z_-$, and using equation~\eqref{eq:sum}, we get
\begin{align*}
 2\pi \m(p_{C_1}) -4\pi \log(1+\sqrt{2})= & -4D\left(-(1+\sqrt{2}) i\right)+4D\left((1+\sqrt{2}) i\right)-4D(-(1-\sqrt{2})i)\\
&+4D((1-\sqrt{2})i)
 +2\log (1+\sqrt{2})\int_{-i}^i d \arg \left(\frac{1+(1+\sqrt{2})w}{1-(1+\sqrt{2})w}\right)\\
&+2\log(\sqrt{2}-1)\int_{-i}^i d \arg \left(\frac{1+(\sqrt{2}-1)w}{1-(\sqrt{2}-1)w}\right).
  \end{align*}
By equation~\eqref{eq:R-integral} in Lemma \ref{lem:R-integral},
\[I(1+\sqrt{2})=\frac{\pi}{2}=-I(\sqrt{2}-1).\]

Thus, by properties \eqref{eq:inverse} and \eqref{eq:conjugate},  we finally obtain
\begin{align*}
 2\pi \m(p_{C_1})=&16D\left((1+\sqrt{2}) i\right)+4\pi \log(1+\sqrt{2})-\pi\log (1+\sqrt{2})+\pi\log (\sqrt{2}-1)\\
=&16D\left((1+\sqrt{2}) i\right)+2\pi \log(1+\sqrt{2}) \approx \ 17.58392561.
 \end{align*}

By \cite[Theorem 3.5]{ckp:bal_vol},
$$ \vol((T^2\times I) -C_1) = \volbp(C_1) = 10\vtet + 2\voct \approx \ 17.47714082.$$
Thus, Conjecture \ref{conj:volbipy} is verified for the link $\CC_1$.
\end{proof} 


\subsection{The family of links $\CC_n$ (numerical results).}

We present some numerical results that generalize the rigorously proven examples $\CC_0$ and $\CC_1$.

\begin{figure}[h]
  \centering 
\includegraphics[height=1.8in]{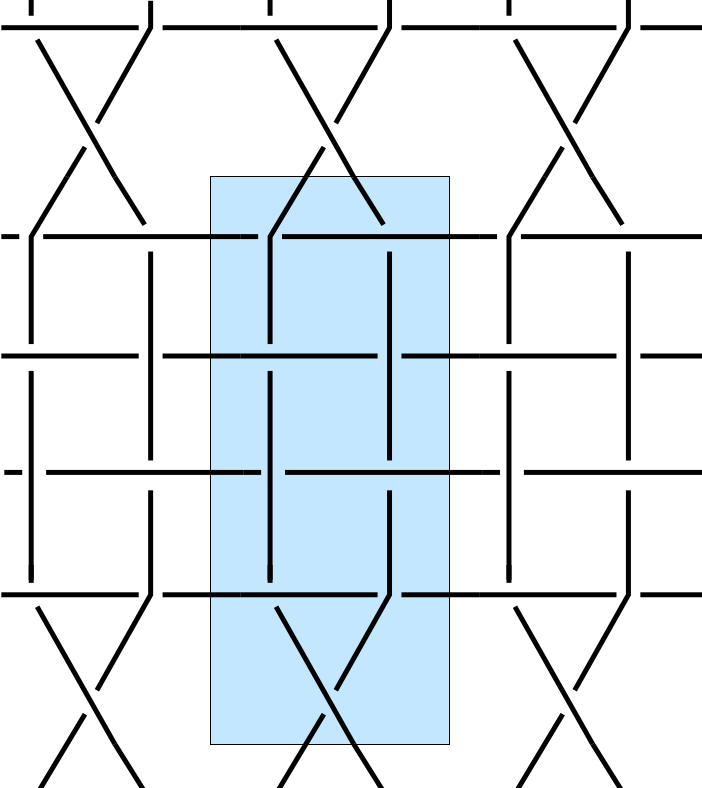}
\caption{Diagram for one of the biperiodic alternating links $\CC_n$, and fundamental
  domain for $C_n$.  The link $C_n$ as shown is for $n=3$. }
  \label{fig:Cn}
\end{figure}

Let $\CC_n$ be the family of biperiodic alternating links shown in
Figure~\ref{fig:c0} ($n=0$), Figure~\ref{fig:c1} ($n=1$), and
Figure~\ref{fig:Cn} ($n=3$).  For even values of $n$, the fundamental
domain like the one shown in Figure ~\ref{fig:Cn} does not result in a
toroidally alternating link.  In these cases, we need to double the
fundamental domain, as in Figure~\ref{fig:c0} for the link $\CC_0$.
Consequently, all the quantities
$c(L),\,\vol(L),\,\volbp(L),\,\m(p(z,w))$ are doubled, which does not
affect the claim in Conjecture~\ref{conj:volbipy}.

\begin{conjecture}
\label{conj:charpol-Cn} 
The characteristic polynomial for the dimer model corresponding to the toroidal link $C_n$ is 
\[p_{C_n}(z,w)=(1+w^2)(1-z)^{n+1}+(-1)^n w \sum_{j=0}^{n+1}\binom{2n+4}{2j+1} z^j.\]
\end{conjecture}

If Conjecture~\ref{conj:charpol-Cn} holds, then Conjecture~\ref{conj:volbipy} would imply that
$ 10 \vtet+2n\voct < 2 \pi \m(p_{C_n})$. 
\[P_n(z,w):=p_{C_n}(z^2,w)=(1+w^2)(1-z^2)^{n+1}+(-1)^n w \frac{((1+z)^{2n+4}-(1-z)^{2n+4})}{2z}.\]
Then $\m(P_n)=\m(p_{C_n})$. 
Computing numerical values for $\m(P_n)$ with Mathematica, we form Table~\ref{table}. We see that Conjecture~\ref{conj:volbipy} is 
numerically confirmed for the first 12 values of $n$. 

\begin{table}
\begin{tabular}{|l|l|l|}
 \hline
 $n$ &  $10 \vtet+2n\voct$ & $2\pi\m(p_{C_n})$ \\
 \hline
 \hline
 $2$ & $24.80486557$ & $24.96932402$ \\ 
 $3$ & $32.13259032$ & $32.27389896$ \\
 $4$ & $39.46031507$ & $39.61527996$ \\
 $5$ & $46.78803983$ & $46.93541034$ \\
 $6$ & $54.11576458$ & $54.26836944$ \\
 $7$ & $61.44348933$ & $61.59270586$ \\
 $8$ & $68.77121409$ & $68.92297116$ \\
 $9$ & $76.09893884$ & $76.2489$ (*)\\
 $10$ & $83.42666359$ & $83.57804426$ \\
 $11$ & $90.75438835$ & $90.9047$ (*)\\
 $12$ & $98.08211310$ & $98.23330183$ \\
 \hline
\end{tabular}

\caption{Values for $\m(p_{C_n})$. For the two values indicated by (*),
we can only get precision up to 4 decimal places.}
\label{table}
\end{table}


\subsection{Medial graph on the 8-8-4 tiling}
Let $\KK$ denote biperiodic alternating link whose projection is the medial graph on the 8-8-4 tiling, as shown in Figure \ref{fig:884}. 
Let $K$ be its alternating quotient link in $T^2\times I$.
In this case, we have 
\[p_K(z,w)=-w^2z^2 + 6w^2z + 6wz^2 - w^2 + 28wz - z^2 + 6w + 6z - 1.\]

\begin{center}
\begin{figure}[h]
\begin{tabular}{cc}
 \includegraphics[height=1.75 in]{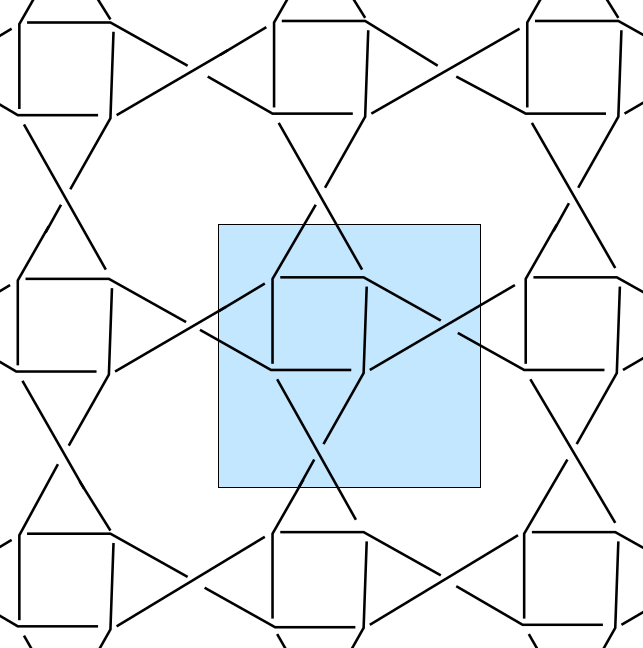} &
\hspace*{2cm}
 \includegraphics[height=1.75 in]{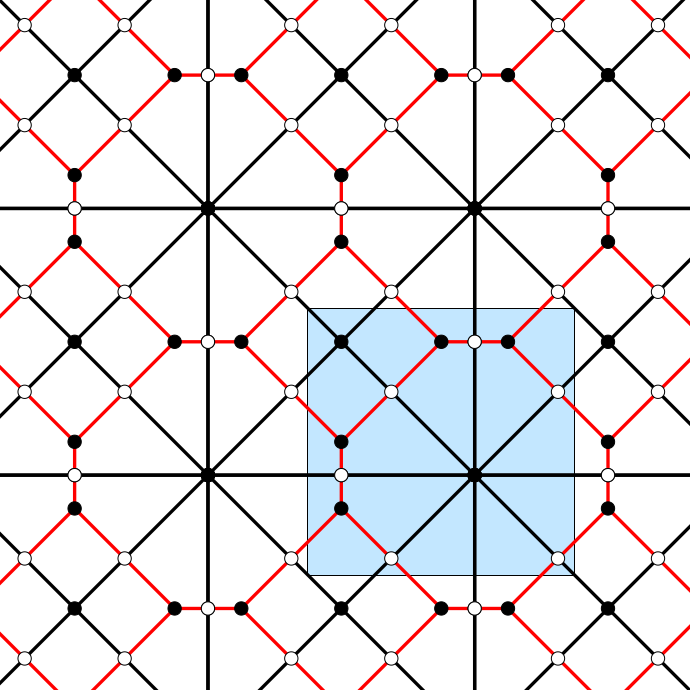}  \\
(a) & \hspace*{2cm} (b)
\end{tabular} 
\caption{(a) Diagram of biperiodic link $\KK$, and fundamental
  domain for $K$.  (b) Overlaid graph $G_{\KK}^b$ and fundamental
  domain for $G_{K}^b$.}
\label{fig:884}
\end{figure}
\end{center}

\begin{theorem}
\label{thm:mm-884}
  $$ 2\pi\, \m(p_K) =  \arccos\left(-\frac{7}{9}\right) \log(17+12\sqrt{2})
+8D(i)+4D\left(\frac{\sqrt{7+4\sqrt{2}i}}{3}\right)-4D\left(-\frac{\sqrt{7+4\sqrt{2}i}}{3}\right).$$
Consequently, $\vol((T^2\times I) - K) < \volbp(K) < 2\pi \,\m(p_K)$.
\end{theorem}

\begin{proof}
The curve defined by the zero locus of this polynomial can be parametrized by 
\begin{align*}
w=&\frac{\sqrt{2}(t^2+1)-\sqrt{3}(t^2-1)}{\sqrt{2}(t^2+1)+\sqrt{3}(t^2-1)}=\frac{(2\sqrt{6}-5)\left(t^2-\left(5+2\sqrt{6}\right)\right)}{t^2-(5-2\sqrt{6})}\\
=&(2\sqrt{6}-5)\frac{\left(t-(\sqrt{3}+\sqrt{2})\right)\left(t+(\sqrt{3}+\sqrt{2})\right)}{\left(t-(\sqrt{3}-\sqrt{2})\right)\left(t+(\sqrt{3}-\sqrt{2})\right)},\\
z=&\frac{\sqrt{2}(t^2+1)-2\sqrt{3}t}{\sqrt{2}(t^2+1)+2\sqrt{3}t}=\frac{t^2-\sqrt{6}t+1}{t^2+\sqrt{6}t+1}=\frac{\left(t-\frac{\sqrt{3}+1}{\sqrt{2}}\right)\left(t-\frac{\sqrt{3}-1}{\sqrt{2}}\right)}{\left(t+\frac{\sqrt{3}+1}{\sqrt{2}}\right)\left(t+\frac{\sqrt{3}-1}{\sqrt{2}}\right)}.
\end{align*}
Setting $z=e^{i\theta}$ we get
\[e^{i\theta}=\frac{t^2-\sqrt{6}t+1}{t^2+\sqrt{6}t+1} \implies 
(e^{i\theta}-1)t^2+(e^{i\theta}+1)\sqrt{6}t+(e^{i\theta}-1)=0 \implies t^2-i \sqrt{6}\cot\left(\frac{\theta}{2}\right)t+1=0.\]
We continuously choose one of the two roots $t$ for the above polynomial in order to obtain the parametrization.
After some numerical computation we conclude that we have
to integrate for $\theta \in (0,\pi)$ and $t$ in the complex
imaginary segment between $0$ and $-i$, and for $\theta \in (\pi,2\pi)$ and 
$t$ in the complex imaginary segment between $i$ and $0$.

The general elements that we need to evaluate are
\[-\alpha_1\alpha_2\alpha_3\alpha_4 \, \left(t+\alpha_1 \sqrt{2}+\alpha_2 \sqrt{3}\right)\wedge \left(t +\frac{1+\alpha_4 \sqrt{3}}{\alpha_3 \sqrt{2}}\right), 
\ {\rm and}\ 
-\alpha_3\alpha_4 \, (2\sqrt{6}-5) \wedge \left(t +\frac{1+\alpha_4 \sqrt{3}}{\alpha_3 \sqrt{2}}\right).\]
where $\alpha_i \in \{\pm 1\}$ (all possible combinations). 

For the terms of the second kind equation \eqref{eq:sum} yields
\[\log|2\sqrt{6}-5|\left . \arg\left(\frac{t^2-\sqrt{6}t+1}{t^2+\sqrt{6}t+1}\right) \right|^{-i}_i=2\pi \log (5-2\sqrt{6})\approx -14.403772983899.\]

For the terms of the first kind we use Lemma \ref{lemma} (and ignore terms of the form $(\pm 1) \wedge x$ and $x \wedge (\pm 1)$) to obtain
\begin{align*}
&-\alpha_1\alpha_2\alpha_3\alpha_4 \, \left(t+\alpha_1 \sqrt{2}+\alpha_2 \sqrt{3}\right)\wedge \left(t +\frac{1+\alpha_4 \sqrt{3}}{\alpha_3 \sqrt{2}}\right)\\
=&-\alpha_1\alpha_2\alpha_3\alpha_4  \, \left(\frac{t+\alpha_1 \sqrt{2}+\alpha_2 \sqrt{3}}{\alpha_1 \sqrt{2}+\alpha_2 \sqrt{3}-\frac{1+\alpha_4 \sqrt{3}}{\alpha_3 \sqrt{2}}}\right) \wedge \left(\frac{t +\frac{1+\alpha_4 \sqrt{3}}{\alpha_3 \sqrt{2}}}{\alpha_1 \sqrt{2}+\alpha_2 \sqrt{3}-\frac{1+\alpha_4 \sqrt{3}}{\alpha_3 \sqrt{2}}}\right)\\
& + \alpha_1\alpha_2\alpha_3\alpha_4 \, \left(\alpha_1 \sqrt{2}+\alpha_2 \sqrt{3}-\frac{1+\alpha_4 \sqrt{3}}{\alpha_3 \sqrt{2}}\right)
\wedge \left(\frac{t+\alpha_1 \sqrt{2}+\alpha_2 \sqrt{3}}{t +\frac{1+\alpha_4 \sqrt{3}}{\alpha_3 \sqrt{2}}}\right).\\
\end{align*}
The terms in the second line integrate to 
\[\alpha_1\alpha_2\alpha_3\alpha_4 \, \log \left|\alpha_1 \sqrt{2}+\alpha_2 \sqrt{3}-\frac{1+\alpha_4 \sqrt{3}}{\alpha_3 \sqrt{2}}\right|
\left. \arg \left(\frac{t+\alpha_1 \sqrt{2}+\alpha_2 \sqrt{3}}{t +\frac{1+\alpha_4 \sqrt{3}}{\alpha_3 \sqrt{2}}} \right)\right|^{-i}_{i}.\]

Notice that exchanging the signs of $\alpha_1$, $\alpha_2$ and $\alpha_3$ together amounts to changing $t=\pm i$ to $t=\mp i$ in the argument 
and does not change the absolute value term inside the logarithm. 
See Table \ref{table:arg-at-t}.

\begin{table}
\begin{center}
\renewcommand{\arraystretch}{1.3}
\begin{tabular}{|cccc|c|c|c|c|}
\hline
&&& & argument & argument  & argument & argument \\
$\alpha_1$ & $\alpha_2$  & $\alpha_3$ & $\alpha_4$  & at $t=i$ &  at $t=0^+$ & at $t=0^-$& at $t=-i$\\
\hline
\hline
$1$ & $1$ & $1$ & $1$ & $\frac{\beta - \pi}{4}$ & $0$ & $0$  & $\frac{\pi-\beta}{4}$ \\
\hline 
$1$ & $-1$ & $1$ & $1$ & $\frac{\pi+\beta}{4}$ & $\pi$ & $-\pi$ & $-\frac{\pi+\beta}{4}$ \\
\hline
$1$ & $1$ & $-1$ & $1$ & $- \frac{3}{4}\pi $ & $-\pi$ & $\pi$ & $\frac{3}{4} \pi$\\
\hline
$1$ & $1$ & $1$ & $-1$ & $\frac{\beta-3\pi}{4}$ & $-\pi$ &$\pi$  &$\frac{3\pi-\beta}{4}$\\
\hline
$1$ & $-1$ & $-1$ & $1$ & $-\frac{1}{4}\pi $ & $0$ & $0$ & $\frac{1}{4} \pi$\\
\hline
$1$ & $-1$ & $1$ & $-1$ & $\frac{\beta-\pi}{4}$ & $0$& $0$  & $\frac{\pi-\beta}{4}$\\
\hline 
$1$ & $1$ & $-1$ & $-1$ & $-\frac{1}{4}\pi $ & $0$ & $0$ &  $\frac{1}{4} \pi$\\
\hline 
$1$ & $-1$ & $-1$ & $-1$ & $\frac{1}{4}\pi $ & $\pi$ &$-\pi$ &  $-\frac{1}{4} \pi$\\
\hline 
\end{tabular}
\renewcommand{\arraystretch}{1}
\end{center}
\caption{Argument at various $t$'s corresponding to values of $\alpha_i$'s. Here $\beta=\arccos\left(-\frac{7}{9}\right)\approx 2.4619188346815493642$. }
\label{table:arg-at-t}
\end{table}

Putting everything together, the integration of the logarithmic terms yields 
\begin{align*}
& \left(\pi -\beta \right)\log \left|\sqrt{2}+\sqrt{3}-\frac{1+\ \sqrt{3}}{\sqrt{2}}\right| 
+(\beta-3\pi) \log \left|\sqrt{2}-\sqrt{3}-\frac{1+\sqrt{3}}{\sqrt{2}}\right|
+\pi \log \left|\sqrt{2}+\sqrt{3}+\frac{1+\sqrt{3}}{\sqrt{2}}\right|\\
&+(\beta+\pi) \log \left|\sqrt{2}+\sqrt{3}-\frac{1-\sqrt{3}}{\sqrt{2}}\right|
+\pi \log \left|\sqrt{2}-\sqrt{3}+\frac{1+\sqrt{3}}{\sqrt{2}}\right|
+\left(\pi -\beta\right)\log \left|\sqrt{2}-\sqrt{3}-\frac{1-\sqrt{3}}{\sqrt{2}}\right|\\
&+\pi \log \left|\sqrt{2}+\sqrt{3}+\frac{1-\sqrt{3}}{\sqrt{2}}\right|
-3\pi \log \left|\sqrt{2}-\sqrt{3}+\frac{1-\sqrt{3}}{\sqrt{2}}\right|.
\end{align*}

The terms containing $\beta$ yield
\begin{align*}
& -\beta\log \left|\sqrt{2}+\sqrt{3}-\frac{1+\ \sqrt{3}}{\sqrt{2}}\right| 
+\beta\log \left|\sqrt{2}-\sqrt{3}-\frac{1+\sqrt{3}}{\sqrt{2}}\right|+\beta \log \left|\sqrt{2}+\sqrt{3}-\frac{1-\sqrt{3}}{\sqrt{2}}\right|\\
&-\beta\log \left|\sqrt{2}-\sqrt{3}-\frac{1-\sqrt{3}}{\sqrt{2}}\right|
= \beta \log (17+12 \sqrt{2}) \approx 8.679480937097002.\\
\end{align*}

The other terms yield
\begin{align*}
& \pi \log \left|\sqrt{2}+\sqrt{3}-\frac{1+\ \sqrt{3}}{\sqrt{2}}\right| 
-3\pi \log \left|\sqrt{2}-\sqrt{3}-\frac{1+\sqrt{3}}{\sqrt{2}}\right|
+\pi \log \left|\sqrt{2}+\sqrt{3}+\frac{1+\sqrt{3}}{\sqrt{2}}\right|\\
&+\pi \log \left|\sqrt{2}+\sqrt{3}-\frac{1-\sqrt{3}}{\sqrt{2}}\right|
+\pi \log \left|\sqrt{2}-\sqrt{3}+\frac{1+\sqrt{3}}{\sqrt{2}}\right|
+\pi \log \left|\sqrt{2}-\sqrt{3}-\frac{1-\sqrt{3}}{\sqrt{2}}\right|\\
&+\pi \log \left|\sqrt{2}+\sqrt{3}+\frac{1-\sqrt{3}}{\sqrt{2}}\right|
-3\pi \log \left|\sqrt{2}-\sqrt{3}+\frac{1-\sqrt{3}}{\sqrt{2}}\right|\\
=&\pi \log(833-588\sqrt{2}-480\sqrt{3}+ 340\sqrt{6})\approx 3.32810583970523.\\
\end{align*}

Finally, the dilogarithm terms are given by 
\[\alpha_1\alpha_2\alpha_3\alpha_4 D \left(\frac{-t -\frac{1+\alpha_4 \sqrt{3}}{\alpha_3 \sqrt{2}}}{\alpha_1 \sqrt{2}+\alpha_2 \sqrt{3}-\frac{1+\alpha_4 \sqrt{3}}{\alpha_3 \sqrt{2}}}\right)\]
Notice that exchanging the signs of $\alpha_1$, $\alpha_2$, and $\alpha_3$ together amounts to changing the sign of $t$. This 
can be combined with formulas in equations~\eqref{eq:inverse} and \eqref{eq:conjugate}  to obtain

\begin{align*}
& 4D  \left(\frac{i -\frac{1+\sqrt{3}}{\sqrt{2}}}{\sqrt{2}+\sqrt{3}-\frac{1+\sqrt{3}}{\sqrt{2}}}\right)
- 4D  \left(\frac{i -\frac{1+\sqrt{3}}{\sqrt{2}}}{\sqrt{2}-\sqrt{3}-\frac{1+\sqrt{3}}{\sqrt{2}}}\right)
 - 4D  \left(\frac{i +\frac{1+\sqrt{3}}{\sqrt{2}}}{\sqrt{2}+\sqrt{3}+\frac{1+\sqrt{3}}{\sqrt{2}}}\right)\\ 
& - 4D  \left(\frac{i -\frac{1-\sqrt{3}}{\sqrt{2}}}{\sqrt{2}+\sqrt{3}-\frac{1-\sqrt{3}}{\sqrt{2}}}\right)
+ 4D  \left(\frac{i +\frac{1+\sqrt{3}}{\sqrt{2}}}{\sqrt{2}-\sqrt{3}+\frac{1+\sqrt{3}}{\sqrt{2}}}\right)
+ 4D  \left(\frac{i -\frac{1- \sqrt{3}}{\sqrt{2}}}{\sqrt{2}-\sqrt{3}-\frac{1-\sqrt{3}}{\sqrt{2}}}\right)\\
&+ 4D  \left(\frac{i +\frac{1-\sqrt{3}}{\sqrt{2}}}{\sqrt{2}+\sqrt{3}+\frac{1-\sqrt{3}}{\sqrt{2}}}\right)
- 4D  \left(\frac{i +\frac{1-\sqrt{3}}{\sqrt{2}}}{\sqrt{2}-\sqrt{3}+\frac{1-\sqrt{3}}{\sqrt{2}}}\right)\\
\approx& 4 \cdot 2.77301284617524  \approx 11.092051384700.\\
\end{align*}

By applying the five-term relation and identities such as the following
\[\frac{-i -\frac{1+\sqrt{3}}{\sqrt{2}}}{\sqrt{2}-\sqrt{3}-\frac{1+\sqrt{3}}{\sqrt{2}}}=
1-\frac{1}{\frac{1-\frac{i +\frac{1-\sqrt{3}}{\sqrt{2}}}{\sqrt{2}-\sqrt{3}+\frac{1-\sqrt{3}}{\sqrt{2}}}}
{1-\frac{i -\frac{1-\sqrt{3}}{\sqrt{2}}}{\sqrt{2}+\sqrt{3}-\frac{1-\sqrt{3}}{\sqrt{2}}}\frac{i +\frac{1-\sqrt{3}}{\sqrt{2}}}{\sqrt{2}-\sqrt{3}+\frac{1-\sqrt{3}}{\sqrt{2}}}}},\]
one can prove
\begin{align*}
&D  \left(\frac{i -\frac{1-\sqrt{3}}{\sqrt{2}}}{\sqrt{2}+\sqrt{3}-\frac{1-\sqrt{3}}{\sqrt{2}}}\right)
+D  \left(\frac{i +\frac{1-\sqrt{3}}{\sqrt{2}}}{\sqrt{2}-\sqrt{3}+\frac{1-\sqrt{3}}{\sqrt{2}}}\right)
=D  \left(\frac{i -\frac{1+\sqrt{3}}{\sqrt{2}}}{\sqrt{2}-\sqrt{3}-\frac{1+\sqrt{3}}{\sqrt{2}}}\right)
+D  \left(\frac{i +\frac{1+\sqrt{3}}{\sqrt{2}}}{\sqrt{2}+\sqrt{3}+\frac{1+\sqrt{3}}{\sqrt{2}}}\right)\\
\end{align*}
and
\begin{align*}
& D  \left(\frac{i -\frac{1+\sqrt{3}}{\sqrt{2}}}{\sqrt{2}+\sqrt{3}-\frac{1+\sqrt{3}}{\sqrt{2}}}\right)
+ D  \left(\frac{i +\frac{1+\sqrt{3}}{\sqrt{2}}}{\sqrt{2}-\sqrt{3}+\frac{1+\sqrt{3}}{\sqrt{2}}}\right)
= D  \left(\frac{i -\frac{1- \sqrt{3}}{\sqrt{2}}}{\sqrt{2}-\sqrt{3}-\frac{1-\sqrt{3}}{\sqrt{2}}}\right)
+ D  \left(\frac{i +\frac{1-\sqrt{3}}{\sqrt{2}}}{\sqrt{2}+\sqrt{3}+\frac{1-\sqrt{3}}{\sqrt{2}}}\right).
\end{align*}

This allows us to simplify the dilogarithm terms as
\begin{align*}
& 8D  \left(\frac{i -\frac{1+\sqrt{3}}{\sqrt{2}}}{\sqrt{2}+\sqrt{3}-\frac{1+\sqrt{3}}{\sqrt{2}}}\right)
- 8D  \left(\frac{i -\frac{1+\sqrt{3}}{\sqrt{2}}}{\sqrt{2}-\sqrt{3}-\frac{1+\sqrt{3}}{\sqrt{2}}}\right)
-8D  \left(\frac{i +\frac{1+\sqrt{3}}{\sqrt{2}}}{\sqrt{2}+\sqrt{3}+\frac{1+\sqrt{3}}{\sqrt{2}}}\right)
+8 D  \left(\frac{i +\frac{1+\sqrt{3}}{\sqrt{2}}}{\sqrt{2}-\sqrt{3}+\frac{1+\sqrt{3}}{\sqrt{2}}}\right).
\end{align*}

By using the identity \eqref{eq:unitary} we see that the dilogarithm terms equal
\[8D(i)+4D\left(\frac{\sqrt{7+4\sqrt{2}i}}{3}\right)-4D\left(-\frac{\sqrt{7+4\sqrt{2}i}}{3}\right).\]

Then we have to add everything as well as the Mahler measure of $z^2-6z+1$ which is
\[2\pi \m(z^2-6z+1)=2 \pi \log(3+2\sqrt{2})\approx 2 \pi\cdot 1.76274717403908 \approx  11.075667144194722.\]

Putting everything together and collapsing terms, we obtain
\begin{align*}
 2\pi \m(p_K)
=& \arccos\left(-\frac{7}{9}\right) \log(17+12\sqrt{2})+8D(i)+4D\left(\frac{\sqrt{7+4\sqrt{2}i}}{3}\right)-4D\left(-\frac{\sqrt{7+4\sqrt{2}i}}{3}\right)\\
\approx & 19.771532321797992256575200922336735211.
\end{align*}

Finally,
$$\volbp(K) = \vol(B_8) + \vol(B_4) + 4 \vol(B_3) \approx 7.8549+3.6638+4 \times 2.0298 =19.6379. $$
Using SnapPy~\cite{snappy} inside Sage to verify the computation rigorously, we verified that 
$$\vol((T^2\times I) - K) \approx 19.559.$$
Thus, the link $\KK$ satisfies Conjecture~\ref{conj:volbipy}, as well as inequality~\eqref{eq:volbp-3} within a range of 0.4\%,
$$ \vol((T^2\times I) - K) < \volbp(K) < 2\pi\, \m(p_K). $$
\end{proof}

We remark that, except for the link $\KK$, the logarithmic terms in
the formulas for $2\pi \m(p)$ for all the other links above are of the
form $q \pi \log (\alpha)$, where $q$ is a rational number and
$\alpha$ is an algebraic number.  In Theorem~\ref{thm:mm-884}, we have
instead a term of the form $\arccos\left(-\frac{7}{9}\right)
\log(\alpha)$. The parameter $-\frac{7}{9}$ is also involved in the
arguments for the dilogarithm terms, since
\[\frac{\sqrt{7+4\sqrt{2}i}}{3}=\exp\left(\frac{i}{2}\left(\pi-\arccos\left(-\frac{7}{9}\right)\right)\right).\]

\bibliographystyle{amsplain} 
\providecommand{\bysame}{\leavevmode\hbox to3em{\hrulefill}\thinspace}
\providecommand{\MR}{\relax\ifhmode\unskip\space\fi MR }
\providecommand{\MRhref}[2]{%
  \href{http://www.ams.org/mathscinet-getitem?mr=#1}{#2}
}
\providecommand{\href}[2]{#2}


\end{document}